\documentclass[reqno,12pt]{amsart}
\usepackage{ragged2e}
\usepackage{amssymb}
\usepackage{amsfonts}
\usepackage{amsmath}
\usepackage{mathrsfs}
\usepackage{diagbox}
\usepackage{color,soul}
\usepackage{bbm}
\usepackage{physics}
\usepackage{hyperref}
\usepackage{amsthm, graphicx,empheq}

\usepackage{framed}

\usepackage{lipsum} 

\usepackage{cleveref}

\newtheorem{lemma}{Lemma}[section]
\newtheorem{theorem}{Theorem}[section]
\newtheorem{proposition}{Proposition}[section]
\newtheorem{corollary}{Corollary}[section]
\theoremstyle{definition}
\newtheorem{definition}{Definition}[section]
\theoremstyle{remark}
\newtheorem{remark}{Remark}[section]
\theoremstyle{example}

\numberwithin{figure}{section}
\numberwithin{equation}{section}
\numberwithin{table}{section}

\newcommand{\mres}{\mathbin{\vrule height 1.6ex depth 0pt width
0.13ex\vrule height 0.13ex depth 0pt width 1.3ex}}

\makeatletter
\@namedef{subjclassname@2020}{\textup{2020} Mathematics Subject Classification}
\makeatother

\begin{document}

\title[Hypersonic limit for Euler flows passing cones]{Hypersonic limit for steady compressible Euler flows passing straight cones}

\author{Qianfeng Li}
\author{Aifang Qu}
\author{Xueying Su}
\author{Hairong Yuan}

\address[Q. Li]{School of Mathematical Sciences,  Key Laboratory of Mathematics and Engineering Applications (Ministry of Education) \& Shanghai Key Laboratory of PMMP,  East China Normal University, Shanghai 200241, China}
\email{\tt qfli@math.ecnu.edu.cn}

\address[A. Qu]{Department of Mathematics, Shanghai Normal University,
Shanghai,  200234,  China}\email{\tt afqu@shnu.edu.cn}

\address[X. Su]{School of Mathematics and Statistics, Xi'an Jiaotong University, Xi'an 710049, China}\email{\tt suxueying789@163.com}

\address[H. Yuan]{School of Mathematical Sciences,  Key Laboratory of Mathematics and Engineering Applications (Ministry of Education) \& Shanghai Key Laboratory of PMMP,  East China Normal University, Shanghai 200241, China}\email{\tt hryuan@math.ecnu.edu.cn}



\date{\today}

\begin{abstract}

  We investigate the hypersonic limit for steady, uniform, and  compressible polytropic gas passing a symmetric straight cone. By considering Radon measure solutions, we show that as the Mach number of the upstream flow tends to infinity, the measures associated with the weak entropy solution containing an attached shock ahead of the cone converge vaguely to the measures associated with a Radon measure solution to the conical hypersonic-limit flow. This justifies the Newtonian sine-squared pressure law for cones in hypersonic aerodynamics. For Chaplygin gas, assuming that the Mach number of the incoming flow is less than a finite critical value, we demonstrate that the vertex angle of the leading shock is independent of the conical body's vertex angle and is totally determined by the incoming flow's Mach number. If the Mach number exceeds the critical value, we explicitly construct a Radon measure solution with a concentration boundary layer.  
   
 \end{abstract}

\allowbreak

\allowdisplaybreaks

\maketitle


\section{Introduction}\label{sec1}

This paper pertains to uniform supersonic flows passing a straight right-circular cone $\mathfrak{C}$, where the cone's half vertex-angle $\theta_0$ is less than the critical detach angle and sonic angle. In this scenario, a conical shock attaches to the cone's vertex, and the flow between the shock-front and the cone is supersonic. In addition to demonstrating the existence of such weak solutions by solving an inverse problem, we are interested in the hypersonic limit and  justify the celebrated Newtonian sine-squared pressure law for cones. To describe the hypersonic-limit flow field, which involves an infinitely-thin boundary layer exhibiting concentration of mass and momentum, an appropriate concept of Radon measure solutions to the compressible Euler equations is required.

We assume that the flow field is governed by the steady compressible Euler system
\begin{equation}\label{E1}
\left\{
\begin{aligned}
& \mathrm{Div}(\rho V)=0, \\
& \mathrm{Div}(\rho V\otimes V)+\mathrm{Grad}\, p=0,\\
& \mathrm{Div}(\rho EV)=0.
\end{aligned}
\right.
\end{equation}
Here $\mathrm{Div}$ and $\mathrm{Grad}$ are the divergence and gradient operator  of the Euclidean space $\mathbb{R}^3$ respectively. The unknowns  $\rho,E$ and $V=(u_1,u_2,u_3)^{\top}$ represent respectively the density of mass, the total enthalpy per unit mass, and the flow's velocity in $\mathbb{R}^3$. For a polytropic gas, the unknown scalar pressure $p$ is given by the state equation
\begin{equation}\label{E2}
p=\frac{\epsilon}{\epsilon+1}\,\rho\cdot(E-\frac{1}{2}|V|^2),
\end{equation}
with $\gamma\doteq\epsilon+1>1$ being the adiabatic exponent. Then the sound speed is $c\doteq\sqrt{\gamma p/\rho}$, and  the Mach number is $M\doteq|V|/c$. It will be shown later that after suitable scalings,  the hypersonic limit is the limit   $\epsilon\to0+$. The case of Chaplygin gas will be considered in Section \ref{sec5}.

The flow is supposed to be given and moves uniformly at supersonic speed from the left half-space:
\begin{equation}\label{Initialdata}
(\rho, V, E)=(\rho_{\infty}, (V_{\infty}, 0, 0)^{\top}, E_{\infty}) \qquad \text{in}\quad
\{x=(x_1,x_2,x_3)\in\mathbb{R}^3 : x_1\le0\}; 
\end{equation}
and is subjected to the slip boundary condition
\begin{equation}\label{Boundarycondition}
V\cdot\vec{n}|_{\mathcal{C}}=0
\end{equation}	
on the boundary 
\begin{align}
\mathcal{C}\doteq\partial\mathfrak{C}=\{x\in\mathbb{R}^3:\, x_1>0, \, x_2^2+x_3^2=x_1^2\tan\theta_0^2\}
\end{align} 
of the cone. Here  $\theta_0\in(0,\, \frac{\pi}{2})$ is the half-open angle of the cone, and 
 $\vec{n}=(n_1,n_2,n_3)$ is the unit  normal vector of $\mathcal{C}$ pointing to the cone,  i.e., $n_1>0$.

We remark that \eqref{E1}-\eqref{Boundarycondition} is an informal formulation of the problem of hypersonic flow passing a cone, which is valid only for classical solutions, while we know such problems do not have classical solutions.  We will reformulate the problem later and present a rigorous form by proposing a concept of Radon measure solutions. Then we have a proper setting for this physical problem.  The other observation is  the state function  \eqref{E2}, which differs from the usual one like $p=\exp(\hat{s})\rho^\gamma$, with $\hat{s}$ being the entropy. It is well-known that to study generalized solutions to hyperbolic conservation laws, one should be very careful to employ the primitive equations  derived directly from the conservation laws. The  choice of \eqref{E1} and \eqref{E2} turns out to be  correct since we can obtain physically consistent results from them. 
It is important to note that, unlike the prohibited change-of-dependent-variables in the state space, any smooth change of independent variables is always permissible. This is because these physical conservation laws do not depend on coordinates in the physical space $\mathbb{R}^3.$

\subsection{Literature review}\label{sec11}
The problem \eqref{E1}-\eqref{Boundarycondition} represents a classic aerodynamic model with wide applications in the aviation industry, including the design of aircraft inlets, waveriders, and aircraft forebodies. Additionally, as a challenging issue in the theory and computation of systems of conservation laws, it has garnered significant mathematical attention.

It has been observed that the problem \eqref{E1}-\eqref{Boundarycondition} admits a special solution with the following properties:

\begin{itemize}
\item[(i)] The state of the flow is constant on each ray emanating from the vertex of the cone. Such a flow is called conical flow; 
\item[(ii)] For a right circular cone with a zero-attack angle, the flow exhibits axial symmetry with respect to the axis of the cone;
\item[(iii)] If the half vertex angle of the cone is smaller than a critical value, which is dependent on the Mach number of the incoming flow, a conical shock-front attaches to the vertex of the cone, with a larger half open angle. The flow ahead of the shock-front remains to be the uniform incoming flow, while the flow behind the shock-front is still supersonic.
\end{itemize}
Extensive mathematical research has been dedicated to the study of these special conical-flow solutions, employing diverse methods and models, and investigating various classes of solutions.

For  isentropic irrotational supersonic flows passing straight cones, Busemann \cite{ABusemann} introduced the famous ``apple curve" with a clever graphical construction in the hodograph plane. Taylor and Maccoll \cite{GTaylor} solved the problem by straightforward numerical integration. It seems that there is no analytical proof of existence of such special conical flows for the non-isentropic compressible Euler equations.   Most of the mathematical analysis  are devoted to the stability  of such  flow patterns under small perturbations of the cone or the incoming flow.

Considering isentropic irrotational supersonic flows past curved cones,  there are significant progresses. For piecewise-smooth solutions with supersonic attached shocks,  see \cite{ChenLiD, CXY, Chen2001, Add2,  Add6, Add4, Add5} and references therein.  Chen and Li \cite{ChenLiD} proved  local stability for the configuration of supersonic conical shock attaching to a  slightly perturbed axisymmetric cone. Chen, Xin, and Yin \cite{CXY} demonstrated that for a high-speed uniform supersonic incoming flow, this particular configuration exhibits global stability.  
Hu and Zhang \cite{Add6} applied characteristic analysis to prove the existence of global piecewise smooth supersonic solutions for flows passing curved cones with large variation, at the cost of requiring  the incoming supersonic flow to be quite dilute. 
For non-axisymmetric perturbations, the local and global stability for such configuration were established in Chen \cite{Chen2001}, and Xin and Yin  \cite{Add2} respectively. 

For piecewise smooth solutions with supersonic shocks under perturbed incoming flows, we  refer to \cite{Add4, Add5}. Cui and Yin \cite{Add4} studied the global existence and asymptotic behavior of the solutions for supersonic flows past straight cones, with symmetric and small perturbations on the incoming flow. Li, Ingo and Yin \cite{Add5} considered general perturbations on the incoming flows. The piecewise smooth solutions with transonic shocks have been considered in \cite{CF,Add3}, in which the authors demonstrated the stability of transonic shock solutions and their far-field asymptotic behavior.

For weak solutions in the class of functions with bounded variations, Lien and Liu \cite{LienLiu}, Wang and Zhang \cite{WangZhang} employed a modified Glimm scheme to establish the global stability and asymptotic behavior of  solutions   for  high Mach number flow past  axially-symmetrically slightly perturbed cones.

For the more general piecewise-smooth isentropic Euler flows, Chen and Li \cite{CL} provided a background solution and demonstrated its local stability under small perturbations applied to both the incoming flow and the surface of the cone.

All the previous works focused on supersonic incoming flows with large Mach numbers. It is natural to ask what happens if the Mach number of the incoming flow goes to infinity.  As indicated in  \cite{Anderson}, the density may go to infinity, and concentration of mass and momentum will occur in an infinite-thin boundary layer. Therefore, the framework of integrable weak solution alone is insufficient to describe the limiting flow field. To address this, we refer to \cite{QUYUAN} and introduce Radon measure solutions in this manuscript to tackle this issue.

Mathematicians have  introduced and studied measure solutions to hyperbolic conservation laws for a long time, see the survey of Yang and Zhang \cite{Hyang} and references therein.  Chen and Liu \cite{GQCHEN} demonstrated the formation of  $\delta$-shocks when the pressure vanishes for the compressible Euler equations of polytropic gases. It is revealed that the vanishing  pressure limit is the high-Mach number limit or hypersonic limit in \cite{QUYUAN}. For compressible Euler equations of a Chaplygin gas, $\delta$-shocks are also necessary to solve certain Riemann problems (see, for example, Kong, Wei, and Zhang \cite{DXKONG}, Sheng, Wang, and Yin \cite{WSHENG}). Qu, Yuan and Zhao \cite{QUYUAN} proved that for steady supersonic flow past straight wedge, as the Mach number of the incoming flow goes to infinity, the piecewise constant weak solution converges to the singular measure solution with density containing a weighted Dirac measure on the surface of the wedge. Qu and Yuan \cite{QY} proposed a definition of Radon measure solutions for uniform hypersonic-limit flow passing a straight cone with attack angle, and constructed a measure solution containing weighted Dirac measures to the case that the attack angle is zero. When the attack angle is nonzero, Qu, Su, and Yuan \cite{QSY} present an algorithm based on Fourier spectral method and Newton's method to solve the derived non-linear and singular ordinary differential equations (ODE), and thus provided a numerical method to solve the problem with high precision.  Jin, Qu, and Yuan \cite{JQY} also studied the Radon measure solutions for hypersonic limit flows passing a finite cylindrically symmetric conical body, and analyzed the interactions between the hypersonic-limit flows and a quiescent gas behind the cones. It is discovered that for certain cases, even $\delta$-shock can only exist within a finite distance from the cone,  and the solution cannot be further extended downstream.

\subsection{Main results and contributions}\label{sec12}

This manuscript aims to answer the following two questions:
\begin{itemize}
	\item[ ] (Q1)~ Is the  problem \eqref{E1},\eqref{Initialdata} and \eqref{Boundarycondition} solvable in the framework of weak entropy solutions for given cone and given high-speed incoming flow?

	\item[ ] (Q2)~ As the Mach number of the incoming flow goes to infinity, does the weak entropy solution converge to the singular measure solution of the hypersonic limit flow in some sense?
\end{itemize}

The answer is summarized in the following theorem. 
\begin{theorem}\label{thm1}
	For uniform incoming supersonic flow past a straight right-circular cone with vertex half-open angle less than $\pi/2$, if the Mach number of the uniform incoming flow $M_{0}$ is suitably large, then there is a piecewise smooth conical weak entropy solution to the problem \eqref{E1}, \eqref{Initialdata} and \eqref{Boundarycondition}, which contains a conical shock attached to the vertex, and the flow ahead of it is the incoming flow, and the flow behind it is supersonic and conical (i.e., constant on each ray issuing from the vertex of the cone).  Furthermore, as $M_{0}\to\infty$,  the measures associated to this weak entropy solution converge vaguely  to  the corresponding  measures of the Radon measure solution of the problem of limiting-hypersonic flow passing the cone, which is given by  \eqref{49} and \eqref{4.10}.
\end{theorem}

For weak entropy solutions and Radon measure solutions of the problem \eqref{E1}, \eqref{Initialdata} and \eqref{Boundarycondition}, see Definition  \ref{Def1} and \ref{Def2}, where the problem is reformulated and reduced to the unit 2-sphere $S^2$, thanks to the fact that the flow is constant on each ray starting from the vertex of the cone. 

Different from Busemann \cite{ABusemann}, or Taylor and Maccoll \cite{GTaylor}, in which supersonic flows passing straight cones were studied by geometric methods or numerical methods, we present an analytical  proof on the existence of conical flow field behind the conical shock.
The conical flow behind the shock is governed by a two-point boundary value problem of nonlinear ODE.
We find from the ODE  that the states $\rho, p, u, w$ are monotonically increasing from the shock to the cone, which ensures the solvability of an inverse problem, and then by checking continuity,  we solve  the direct problem. Also, the monotonicity of these quantities helps us to demonstrate that the shock surface converges to the  cone's surface as  the incoming flow's Mach number tends to infinity, which is the key point for proving that the piecewise smooth weak entropy solution converges to the singular measure solution in the sense indicated in Theorem \ref{thm1} (or to be more specific, see Lemma \ref{lem41}). The convergence also demonstrates  the rationality of the concept of  Radon measure solutions, and justifies the Newtonian sine-squared pressure  law in hypersonic aerodynamics.

\subsection{Structure of the paper} The rest of the paper is organized as follows. In Section \ref{sec2}, we firstly reformulate our problem in spherical coordinates and present the definitions of weak entropy solution and Radon measure solution. Then, assuming that the flow field is piecewise smooth and axially symmetric, we reduce the original problem to a boundary value problem of ODE. In Section \ref{s3}, we apply the shooting method to study the solvability of the ODE problem. Indeed, we firstly study the inverse problem for hypersonic flow past cones, namely problem \eqref{E1}, \eqref{Initialdata} and \eqref{Boundarycondition} with given incoming flow and leading shock position, i.e., it's half-open angle $\beta$. Then we prove that the mapping from the shock position $\beta$ to the cone's surface (represented by the cone's half-open angle $\theta_0$) is continuous, and  solve the direct problem, namely problem \eqref{E1}, \eqref{Initialdata} and \eqref{Boundarycondition} with given incoming flow and $\theta_0$. In Section \ref{sec4}, we prove that the measures of the piecewise-smooth weak solutions converge vaguely to the measures associated with the  Radon measure solution of the hypersonic-limit-flow problem. Section \ref{sec5} is devoted  to the case that  a supersonic Chaplygin gas passes straight cones.

\section{Conical flow and reformulation of the problem}\label{sec2}

As in \cite{QY}, we divide $V$ by $V_\infty$, $\rho$ by $\rho_\infty$, $p$ by $\rho_\infty|V_\infty|^2,$ $E$ by $|V_\infty|^2$, and still denote the acquired functions as $ V, \, \rho,\,  p,$ and $E$. It is obvious that they also satisfy the system  \eqref{E1}, but the incoming flow becomes
\begin{equation}\label{2.5}
    U_0=(\rho_0=1,~ V_0=(1, 0, 0)^\top, ~E_0),~~~\quad E_0>\frac{1}{2}. 
\end{equation}
We immediately get
\begin{equation}\label{2.6}
    p_0=\frac{\epsilon}{\epsilon+1}(E_0-\frac{1}{2}), 
\end{equation}
and
\begin{equation}\label{2.7}
    M_0^2=\frac{V_0^2}{c_0^2}=\frac{1}{c_0^2}=\frac{1}{\epsilon(E_0-\frac{1}{2})}. 
\end{equation}
Thus $\epsilon\to0$ is equivalent to $M_0\to\infty$ (i.e., the hypersonic limit).\footnote{Here $E_0>1/2$ is an assumption on the incoming flow to ensure that $\epsilon\to0$ implies $M_{0}\to+\infty$, see also  \cite[Proposition 1]{QUYUAN}. } In the following we also set $E'\doteq E_0-\frac{1}{2},$ and consider 
the following  problem: 
\begin{framed}
\textsc{Problem 1:}~ Find a solution in an appropriate sense  to \eqref{E1},\eqref{2.5},\eqref{2.6}, and \eqref{Boundarycondition}.
\end{framed}

\subsection{Conical flow and reduced problem  on $S^2$}
Notice that if
$$U(x)=(\rho(x),V(x), E(x))\in \mathbb{R}^5$$
solves Problem 1,  then for any fixed positive number  $a$, $U(ax)$ also solves it. Thus Problem 1 admits solutions $U$ satisfying $U(x)=U(ax)$ for all $a>0$. Such flows are called {\it conical flows.} They are constant on any ray issuing from the origin, {i.e.}, vertex of the cone.   
Therefore,  for conical flows, the problem can be reformulated  on the unit sphere $S^2\subset\mathbb{R}^3$. Applying  some notations and concepts from differential
geometry as in \cite{QY},  we introduce the decomposition into the tangential and radial part of the velocity $V$: \begin{equation}
    V=\tilde{u}+w\vec{\partial}_r, \quad\text{and}\quad \tilde{u}=u\vec{\partial}_\theta+v\vec{\partial}_\phi.
\end{equation}
Here, $\vec{\partial}_\theta, \,\vec{\partial}_\phi,\,\vec{\partial}_r$ are the basis vectors in the latitude, longitude and radial direction under the natural frame of spherical coordinates of $\mathbb{R}^3$, namely $r=|x|$, $x_1=r\cos\theta$, $x_2=r\sin\theta\cos\phi$, $x_3=r\sin\theta\sin\phi,$ and $\theta\in[0,\,\pi]$, $\phi\in[0,\, 2\pi)$.   Referring to \cite{QY}, the governing compressible Euler equations on $S^2$ read:
\begin{align}
	& \mathrm{div}(\rho \tilde{u})+2\rho w=0,~ \label{eq1}\\
	& \mathrm{div}(\rho \tilde{u}E)+2\rho wE=0,~ \label{eq2}\\
	& \mathrm{div}(\rho w\tilde{u})+2\rho w^2-\rho|\tilde{u}|^2=0 \label{eq3},~\\
	& \mathrm{div}(\rho V\otimes V)+3\rho w\tilde{u}+\mathrm{grad}\, p=0,  \label{eq4}
\end{align}
where 
\begin{align}
    & \mathrm{div}\,\tilde{u}\doteq\frac{1}{\sqrt{\sin\theta}}\left(\partial_\theta(\sqrt{\sin\theta}\,u)+\partial_\phi(\sqrt{\sin\theta}\,v) \right),\\
    & \mathrm{grad}\,p\doteq \partial_\theta p\cdot\Vec{\partial}_\theta+\sin^2\theta\partial_\phi p \cdot\Vec{\partial}_\phi
\end{align}
are respectively the divergence and gradient operator on $S^2$. The cone's surface is now reduced to the circle $C$ on $S^2$ given by  
\begin{align}
{C}\doteq\{(r, \theta,\,\phi):\,r=1,\, \theta=\theta_0,\, \phi\in[0,\,2\pi)\}.\end{align} 
The slip condition on ${C}$ becomes
\begin{equation}
(\tilde{u},\,\mathbf{n})\Big|_{{C}}=0,  \label{NewBoundarycondition}
\end{equation}
with $\mathbf{n}$ the unit outer normal vector of ${C}$ on $S^2$, and $(\cdot,\cdot)$ being the inner product of tangent vectors on ${S}^2.$ 

Set  $\Omega\doteq\{(r, \theta,\phi):\, r=1,\,\theta\in(\theta_0,\,{\pi}],\, \phi\in[0,\,2\pi)\}$ to be the region occupied by the gas on the sphere $S^2$,
with the uniform flow 
\begin{align}
U_0=(\rho_0=1, \,V_0=u_0\vec{\partial}_\theta+v_0\vec{\partial}_\phi+w_0\vec{\partial}_r=&\sin\theta\vec{\partial}_\theta+\cos\theta\vec{\partial}_r,\, E_0),   \nonumber\\
&\quad\quad \theta\in(\beta, \,\pi], ~~\phi\in[0,2\pi)\label{upcoming}
\end{align}
 ahead of the leading conical shock-front, whose half vertex angle is $\beta\in(\theta_0, \,\pi/2)$. If  $\theta_0$ is fixed, then $\beta$ depends only on the $\epsilon$ in the state equation \eqref{E2}.

Hence \textsc{Problem $1$} is reformulated as the following 
\begin{framed}
 \textsc{Problem $1'$:}  Find a solution in an appropriate sense to  \eqref{eq1}-\eqref{eq4}, \eqref{NewBoundarycondition}, \eqref{upcoming}.
\end{framed}

\subsection{Definitions of weak entropy solution and Radon measure solution}
In this subsection we present a rigorous understanding of Problem $1'$ by introducing the definitions of weak entropy solutions and Radon measure solutions.

\begin{definition}[weak entropy solution]\label{Def1}
    Let $\rho,~w,~E$ be bounded integrable functions, and $\tilde{u}$ a bounded integrable vector field on $\Omega$. We call $(\rho,\tilde{u},w,E)$ a weak entropy solution to \textsc{Problem $1'$}, if for any $\psi\in C^1(S^2)$, and any $C^1$ vector field $\Psi$ on $S^2,$ 
    \begin{align}
        \int_\Omega(\rho \tilde{u},\,\nabla\psi)\,\mathrm{d}A & =\int_\Omega 2\rho w\psi\, \mathrm{d}A,~ \label{3.2.1}\\
       \int_\Omega(\rho E\tilde{u},\,\nabla\psi)\,\mathrm{d}A & =\int_\Omega 2\rho E w\psi \,\mathrm{d}A,~ \label{3.2.2}\\
       \int_\Omega(\rho w\tilde{u},\,\nabla\psi)\,\mathrm{d}A & =\int_\Omega (2\rho w^2-\rho|\tilde{u}|^2)\psi \,\mathrm{d}A,~ \label{3.2.3}\\
       \int_\Omega(\rho \tilde{u}\otimes \tilde{u},\,\mathrm{D}\Psi)\,\mathrm{d}A+\int_\Omega p~\mathrm{div}\Psi\,\mathrm{d}A & =\int_\Omega 3\rho w(\tilde{u},\,\Psi)\,\mathrm{d}A+\int_{{C}} p(\mathbf{n},\Psi)\,\mathrm{d}s,~\label{3.2.4}
   \end{align}
   and the pressure $p$ increases for flow crossing any discontinuity in $(\rho, \tilde{u}, w, E)$. Here and in the sequel,  $\mathrm{d}A$ is the standard area measure on $S^2$, and $\mathrm{d}s$ is the arc-length differential of the curve $C\subset S^2$. The definitions of $\nabla$ and $\mathrm{D}$ (the gradient and covariant differential on $S^2$) refers to \cite[Table 3]{QY}.\footnote{These definitions are not necessary for the main part of this work.} 
\end{definition}

\begin{definition}[Radon measure solution] \label{Def2}
	Let $\varrho$ be a nonnegative Radon measure on $\bar{\Omega}$, and $w,\,E$ be $\varrho$-measurable scalar functions, and $\tilde{u}$ be a $\varrho$-measurable vector field. We call $(\varrho,\,\tilde{u},\,w,\,E)$ a Radon measure solution to Problem $1'$, if there exist Radon measures $m_a, m_e, m_r, m_t, n_a, n_e, n_r, n_t, \wp$ on $\bar{\Omega}$, and a  positive integrable function $W_C$ on ${C}$, such that 
	
	(i) $m_a, m_e, m_r, m_t, n_a, n_e, n_r, n_t, \wp$ are absolute-continuous with respect to $\varrho$, and $\wp$ is nonnegative;
 	
	(ii)  for any  $\psi\in C^1(S^2),$ and any continuously differentiable vector field $\Psi$ on $S^2$,
	\begin{align}
	& \langle m_a,\,\nabla \psi  \rangle=2 \langle n_a,\,\psi  \rangle ,\\
	&\langle m_e,\,\nabla \psi  \rangle=2\langle n_e,\,\psi  \rangle,\\
	&\langle m_r,\,\nabla\psi  \rangle=2\langle n_r,\,\psi  \rangle-\langle n_t,\,\psi  \rangle,\\
	&\langle m_t,\,\mathrm{D}\Psi\rangle+\langle \wp,\, \mathrm{div}\Psi  \rangle=3\langle m_r,\,\Psi  \rangle+\langle W_C\mathbf{n}_C\delta_C,\, \Psi  \rangle;
	\end{align}

	(iii) the Radon-Nikodym derivatives fulfill the following nonlinear constraints: 
		\begin{align}
	& E=\frac{\mathrm{d}m_e/\mathrm{d}\varrho}{\mathrm{d}m_a/\mathrm{d}\varrho}=\frac{\mathrm{d}n_e/\mathrm{d}\varrho}{\mathrm{d}n_a/\mathrm{d}\varrho},\\
	& w=\frac{\mathrm{d}m_r/\mathrm{d}\varrho}{\mathrm{d}m_a/\mathrm{d}\varrho}=\frac{\mathrm{d}n_r/\mathrm{d}\varrho}{\mathrm{d}n_a/\mathrm{d}\varrho}=\frac{\mathrm{d}n_a}{\mathrm{d}\varrho},\\
	& \tilde{u}=\frac{\mathrm{d}m_a}{\mathrm{d}\varrho},~~~~~|\tilde{u}|^2=\frac{\mathrm{d}n_t}{\mathrm{d}\varrho},~~~~~\tilde{u}\otimes \tilde{u}=\frac{\mathrm{d}m_t}{\mathrm{d}\varrho};
	\end{align}
	
	(iv) if $\varrho\ll\mathrm{d}A$, then for the  Radon-Nikodym derivations 
	\begin{equation}
	\rho=\frac{\mathrm{d}\varrho}{\mathrm{d}A},~~ p=\frac{\mathrm{d}\wp}{\mathrm{d}A},
	\end{equation}
 it holds $\mathrm{d}A$-a.e. that 
	\begin{equation}\label{eq226new}
	p=\frac{\epsilon}{\epsilon+1}\rho\cdot(E-\frac{1}{2}(|\tilde{u}|^2+w^2)),~
	\end{equation}
	and the physical entropy condition is valid, namely $p$ increases for flows across any discontinuity in $(\rho,\, \tilde{u},\, w,\, E)$. \qed
\end{definition}

Here and in the sequel, the pairing $\langle\cdot,\cdot\rangle$ is defined by  $$\langle m,\phi\rangle\doteq\int_{S^2} \phi\cdot
	\mathrm{d}m,$$ where $m$ is a (vector-valued) Radon measure on $\bar{\Omega}\subset S^2$, and $\phi \in C(S^2)$ is a (vector-valued)  test function. Particularly, if $s$ is the arc-length parameter of the curve ${C}\subset S^2$, then 
$$	\langle W_C\mathbf{n}_C\delta_C,\, \Psi  \rangle=\int_{{C}}W_C(s)(\mathbf{n}_C(s), \,\Psi|_{{C}}(s))\,\mathrm{d} s.$$   
We also use standard notations in measure theory, such as $m\mres \Omega$ to denote the measure obtained by restricting a measure $m$ to a $m$-measurable set $\Omega$,  and $\varrho\ll\mathrm{d}A$ means the measure $\varrho$ is absolute-continuous with respect to $\mathrm{d}A$.

As shown in \cite{QY}, the function $W_C$ represents the force per unit area (i.e., pressure) acting on the cone by the fluids.

\begin{remark}
	Radon measure solution is a generalization of the weak entropy solution. In fact, supposing that $(\rho, \tilde{u}, w, E)$ is a weak entropy solution to   \textsc{Problem $1'$}, then ${U}\doteq(\varrho\doteq\rho\, \mathrm{d}A, \tilde{u}, w,E)$ is a Radon measure solution to {Problem $1'$}.
\end{remark}

\subsection{Further simplification of {Problem $1'$} }
According to Anderson \cite{Anderson}, the flow field of  Problem $1'$ with finite incoming Mach number is composed of constant state ahead of a shock and a smooth self-similar field behind the shock. Due to the rotational invariance of the problem  with respect to the axis of the cone, the solution is independent of $\phi$, and $v\equiv0$; that is, the weak entropy solution is of the form
\begin{equation}\label{ss}
U=
\left\{
\begin{aligned}
&  (\rho_0=1, \, V_0\in S^2,\, E_0)\doteq U_0,& \text{if}~\beta<\theta\le \pi, \\
& (\rho(\theta), \, \tilde{u}=u(\theta)\vec{\partial}_{\theta}, \, w(\theta), \,E(\theta)), &\text{if}~ \theta_0<\theta<\beta.
\end{aligned}
\right.
\end{equation}
Recall here that $\theta_0$ is the half vertex angle of the given cone, and $\beta$ is the half-vertex-angle of the corresponding conical shock-front. 
Without confusion, in the following, we denote $U=U(\theta)$,  and the notation $f'$ means  $f'(\theta)$.

\begin{proposition}\label{Prop1}
	Suppose that $U$ is a weak entropy solution of the form  \eqref{ss} and $C^1$ in $(\theta_0, \beta).$ Then for $\theta\in(\theta_0, \beta)$, it holds that  
	\begin{equation}\label{Equ1}
	\left\{
	\begin{aligned}
	& \rho u \cot\theta+\partial_\theta (\rho u) +2\rho w  =0,\\
	& \partial_\theta w=u,\\
	& \rho w u+\rho u\partial_\theta  u +\partial_\theta  p =0,\\
	& \partial_\theta E  =0,
	\end{aligned}
	\right.
	\end{equation}
	with the boundary conditions 
	\begin{equation}\label{Shockcondition}
		\begin{aligned}
	&\rho(\beta)=\frac{(\epsilon+2)M_{0n}^2}{2+\epsilon M_{0n}^2} \rho_0,\quad  u(\beta)=-\frac{2+\epsilon M_{0n}^2}{(\epsilon+2)M_{0n}^2} \sin\beta,\\
	&p(\beta)=\frac{2(\epsilon+1)M_{0n}^2-\epsilon}{2+\epsilon}p_0, \quad M_n^2(\beta)=\frac{M_{0n}^2+\frac{2}{\epsilon}}{\frac{2(\epsilon+1)}{\epsilon}M_{0n}^2-1},\\
	& E(\beta)= E_0,\quad w(\beta)=\cos\beta,
	\end{aligned}
	\end{equation}
	and
	\begin{equation}\label{Conecondition}
	u(\theta_0)=0.
	\end{equation}
Here $M_n\doteq{|u|}/{c}$ denotes the tangential Mach number, and $M_{0n}=u_0(\beta)/c_0=M_0\sin\beta,$ with $M_0$  given by \eqref{2.7}. 

\end{proposition}

\begin{proof}
	Since $U$ is piecewise smooth, Definition \ref{Def1} shows that \eqref{eq1}-\eqref{eq4} hold in the region between the cone's surface and the assumed shock-front  $\{(r,\theta,\phi):\, r=1,\, \theta\in(\theta_0, \beta)\}$, and the Rankine-Hugoniot jump conditions and entropy condition  hold on the shock-front $\{(r,\theta,\phi):\, r=1,\, \theta=\beta\}.$ Substituting \eqref{ss} in \eqref{3.2.1}-\eqref{3.2.4}, direct computations give \eqref{Equ1}, and the Rankine-Hugoniot jump conditions read 
	\begin{align}
	& (\rho_0u_0-\rho u,~\vec{\partial}_\theta)\Big|_{\theta=\beta}=0, \label{RH1.1}\\
	& (\rho_0u_0E_0-\rho uE,~\vec{\partial}_\theta)\Big|_{\theta=\beta}=0,\\
	& (\rho_0w_0u_0-\rho wu,~\vec{\partial}_\theta)\Big|_{\theta=\beta}=0,\\
	& (p+\rho(u,\,\vec{\partial}_\theta)^2)\Big|_{\theta=\beta}=p_0+\rho_0(u_0,\,\vec{\partial}_\theta)^2\Big|_{\theta=\beta},\\
	& \rho_0(u_0,\,\vec{\partial}_\theta)(u_0,\,\vec{\partial}_\phi)\Big|_{\theta=\beta}+\rho(u,\,\vec{\partial}_\theta)(u,\,\vec{\partial}_\phi)\Big|_{\theta=\beta}=0. \label{RH1.5}
	\end{align}	The entropy condition is 
	\begin{align}p(\beta^{\epsilon})>p_0. \label{RH1.6}
	\end{align}
	From these algebraic equations and inequality, one may get   \eqref{Shockcondition} by straightforward calculations, and \eqref{Conecondition} follows directly from \eqref{NewBoundarycondition}.
	The proof is complete.
	\end{proof}

	We next  derive some formulas from Proposition \ref{Prop1}, which is used to prove Lemma \ref{Lemma111} in the next section. 
	
	\begin{corollary}\label{Cor1}
	Suppose that $(\rho(\theta), u(\theta), w(\theta), E(\theta))$ is a smooth solution of \eqref{Equ1}, and $p$ is given by \eqref{E2}. Then we have 
	\begin{equation}\label{A13}
	(\frac{p}{\rho^{\epsilon+1}})'(\theta)=0, \qquad\theta\in(\theta_0,\beta),
	\end{equation}
	and
	\begin{equation}\label{A14}
	(|V|^2)'(\beta)>0.
	\end{equation}
	\end{corollary}

    \begin{proof} 
    	By direct computation, the first and the third equation in \eqref{Equ1} yield that 
    	\begin{equation}\label{A5}
    	\frac{\rho'}{\rho}=-\frac{u'+u\cot\theta+2w}{u}
    	\end{equation}
    	and 
    	\begin{equation}\label{A6}
    \frac{p'}{\rho}=-(uw+uu')
    	\end{equation}
    	respectively.
    	By the forth equation in \eqref{Equ1} and the fifth equation in \eqref{Shockcondition}, \eqref{E2} is reduced to 
    	\begin{equation}\label{A15}
    	p=\frac{\epsilon}{\epsilon+1}\,\rho\,(E_0-\frac{1}{2}|V|^2).
    	\end{equation} 
    	Then, dividing \eqref{A15} by $\rho$, and differentiating it, one has
    	\begin{equation}\label{A4}
    	\left\{
    	\begin{aligned}
    	&\frac{p}{\rho}=\frac{\epsilon}{1+\epsilon}(E_0-\frac{1}{2}|V|^2),\\
    	&(\frac{p}{\rho})'=\frac{p'}{\rho}-\frac{p}{\rho}\frac{\rho'}{\rho}= -\frac{\epsilon}{\epsilon+1}u(u'+w),
    	\end{aligned}
    	\right.
    	\end{equation} 
    	where we have used the fact that $w'=u$.  Thus, by \eqref{A6} and \eqref{A4}, direct computation shows that 
    	\begin{equation}
    	(\frac{p}{\rho^{\epsilon+1}})'=\frac{1}{\rho^{\epsilon}}(\frac{p'}{\rho}-(1+\epsilon)\frac{p}{\rho}\frac{\rho'}{\rho})=0.
    	\end{equation} 
	
To prove \eqref{A14}, notice that the second equation in \eqref{Equ1} implies that \begin{align}\label{eq243new}
(|V|^2)'=ww'+uu'=u(w+u'),\end{align} while $u(\beta)$ and $w(\beta)$ are given by \eqref{Shockcondition}, thus what left is to calculate $u'(\beta).$ 
Substituting \eqref{A5}, \eqref{A6}, and the first equation of \eqref{A4} into the second equation of \eqref{A4}, we obtain  \begin{equation}
  u'=\frac{wu^2-\epsilon/2(2E_0-u^2-w^2)(u\cot\theta+2w)}{\epsilon/2(2E_0-u^2-w^2)-u^2},
  \end{equation} 
  which helps us to show that 
  \begin{equation}\label{A7}
  \begin{aligned}
  (|V|^2)'=u(w+u')&=u\frac{-\epsilon/2(2E_0-u^2-w^2)(w+u\cot\theta)}{\epsilon/2(2E_0-u^2-w^2)-u^2}\\&=u\frac{-c^2(w+u\cot\theta)}{c^2-u^2}\\&=u\frac{-(w+u\cot\theta)}{1-M^2_n}.
  \end{aligned}
  \end{equation}
  
Next we set $\theta=\beta$. By \eqref{Shockcondition}, we see
  \begin{equation}\label{A8}
  \begin{aligned}
  w(\beta)+u(\beta)\cot\beta=\frac{2(M^2_{0n}-1)\sin\beta}{(\epsilon+2)\tan\beta M^2_{0n}},
  \end{aligned}
  \end{equation}
  and 
  \begin{equation}\label{A9}
  \begin{aligned}
  1-M_n(\beta)^2=\frac{(1+2/\epsilon)(M^2_{0n}-1)}{2(1+1/\epsilon)M^2_{0n}-1}.
  \end{aligned}
  \end{equation}
Substituting \eqref{A8} and \eqref{A9} into \eqref{A7} yields
  \begin{equation}
  (|V|^2)'(\beta)=-u(\beta)\frac{2\epsilon\sin\beta(2(1+1/\epsilon)M^2_{0n}-1)}{(\epsilon+2)^2\tan\beta M^2_{0n}}>0,
  \end{equation}
  where $2(1+1/\epsilon)M^2_{0n}-1>0$ is given by the forth equation in \eqref{Shockcondition}, and $-u(\beta)>0$ is guaranteed by the second equation in \eqref{Shockcondition}. The proof is complete.
\end{proof}

\begin{remark}
Observing that $(\ref{Equ1})_4$ and \eqref{Shockcondition} imply that $E\equiv E_0$, then by \eqref{eq226new}, for fixed $\epsilon>0$, we have the following upper bound:
\begin{align}\label{eq250}
 \rho^{\gamma-1}<\frac{\epsilon E_0 \rho^{\gamma}({\beta})}{(\epsilon+1)p({\beta})}, ~~c^2<\epsilon E_0, ~~w<2E_0.\end{align} 
\end{remark}

In the following section, we focus on \eqref{Equ1}, \eqref{Shockcondition} and \eqref{Conecondition}. For clarity, we will refer to it as the ``direct problem" when $\theta_0$ is given and $\beta$ is to be determined, and as the ``inverse problem" when $\beta$ is given and $\theta_0$ is to be determined.  Both the direct and the inverse problems depend on the parameter $\epsilon>0$, or equivalently, on the Mach number $M_0$ of the incoming flow, cf. \eqref{2.7}.

\section{Existence of weak entropy solutions for supersonic flows past cones}\label{s3}
In this section, we aim to  solve the  direct problem by firstly considering the  inverse problem, and then studying the continuity of the map from the shock-front's half-vertex angle $\beta$ to the cone's half-vertex angle $\theta$. 

\subsection{Solvability of the inverse problem}
\subsubsection{Monotonicity of the flow behind  of shock} \label{s3.1}
   
\begin{lemma}\label{Lemma111}
	For the inverse problem with given shock-front $\theta=\beta_0$, and $\epsilon\in(0,\frac{\sin^2\beta_0}{E'}),$ suppose that there exist $\rho,\, u,\, w,\, E\in C^1((\beta_0-\kappa,\, \beta_0])$ for some $\kappa\in(0,\, \beta_0)$, with $\rho>0, \, u\neq 0,\,  M_n\doteq{|u|}/{c}\neq1$, and they solve \eqref{Equ1}.  Then $\rho$ is monotonically decreasing in $(\beta_0-\kappa,\, \beta_0).$ 
\end{lemma}

\begin{proof}
	We first clarify that $\epsilon\in(0,\, \frac{\sin^2\beta_0}{E'})$ implies $M_{0n}>1.$ Indeed, by the definition of Mach number, from \eqref{2.6}, we see 
	$$E'\doteq E_0-\frac{1}{2}=\frac{1}{\epsilon M_{0}^2}.$$ 
Thus, by the definition  $M_{0n}\doteq M_0\sin\beta_0,$ it follows that $$M_{0n}^2=\frac{\sin^2\beta_0}{E'\epsilon},$$ which gives that $M_{0n}>1$ if $\epsilon\in(0,\frac{\sin^2\beta_0}{E'}).$ Furthermore, it follows from \eqref{A9} that
	\begin{equation}\label{AA6}
	M_n(\beta_0)<1.
	\end{equation}

	We next use reduction to absurdity to prove that \begin{align}\label{eq31new}\frac{\mathrm{d}|V|^2}{\mathrm{d}\theta}>0,\quad \theta\in(\beta_0-\kappa,\beta_0).\end{align}  In fact, recalling \eqref{A14} (with $\beta$ being $\beta_0$ here) and supposing that there exists a $\theta_1\in(\beta_0-\kappa,\beta_0)$ such that 
	\begin{equation}
	\left\{
	\begin{aligned}
	&(|V|^2)'(\theta)>0,\qquad \theta\in(\theta_1,\beta_0],\\ &(|V|^2)'(\theta_1)=0,
	\end{aligned}
	\right.
	\end{equation}
	 then \eqref{eq243new} gives that 
	\begin{equation}\label{A10}
	u(\theta_1)(w(\theta_1)+u'(\theta_1))=0.
	\end{equation}
	By the condition that $u(\theta)\neq0, M^2_n\neq 1$ for $\theta\in(\beta_0-\kappa,\beta_0),$ we derive from \eqref{A7} that 
	\begin{equation}
	w(\theta_1)+u(\theta_1)\cot\theta_1=0,
	\end{equation}
	which further gives that at $\theta=\theta_1,$ the velocity is along the $x_1$-direction; that is,  
	\begin{equation}\label{A11}
	q^{\perp}(\theta_1)=0,
	\end{equation}
	with
	\begin{equation}\label{A12}
	q^{\perp}(\theta)\doteq w\sin\theta+u\cos\theta.
	\end{equation} 
	From \eqref{Shockcondition}, it is easy to check that $q^\perp(\beta_0)>0$ provided  $M_{0n}>1$. 
However, thanks to $(\ref{Equ1})_2$, direct computation shows that for $\theta\in(\theta_1,\beta_0),$
	\begin{equation}
	\begin{aligned}
	(q^{\perp})'(\theta)&=(w(\theta)+u'(\theta))\cos\theta\\&=\frac{(|V|^2)'(\theta)\cos\theta}{u(\theta)}<0,
	\end{aligned}
	\end{equation} 
	where we use the fact that $u(\theta)<0$ for $\theta\in(\beta_0-\kappa,\beta_0),$ which is derived from the assumptions that $u(\theta)\in C^1((\beta_0-\kappa,\,\beta_0]), \, u\neq0$, and the fact that  $u(\beta_0)<0$ (cf. \eqref{Shockcondition}).  Thus, we have 
	\begin{equation}
	q^{\perp}(\theta_1)>q^{\perp}(\beta_0)>0,
	\end{equation}
	which is a contradiction to \eqref{A11}. So we proved \eqref{eq31new}. 
	
	Differentiating \eqref{A15} and applying the fact $(\frac{p}{\rho^{\epsilon+1}})'=0$ in Corollary \ref{Cor1}, we have 
	\begin{equation}\label{AA4}
	2(1+\epsilon)\frac{p}{\rho^2}\rho'+(|V|^2)'=0,
	\end{equation}
	which, with \eqref{eq31new}, gives that 
	\begin{equation}
	\rho'(\theta)<0,\quad\ \theta\in (\beta_0-\kappa,\beta_0].
	\end{equation}
Then the proof is complete.
\end{proof}

\begin{corollary} \label{Corollary111}
Suppose that $(\rho,\, u, \,w,\, E),\, \epsilon,\, \beta_0$ is the same as in Lemma \ref{Lemma111}. Then $w,\, c\doteq\sqrt{\gamma p/\rho}$ and $u$ decrease in $(\beta_0-\kappa,\, \beta_0]$, and $M_n$ increases in $(\beta_0-\kappa,\, \beta_0]$. 
\end{corollary}
\begin{proof}
	1.~Since $u(\theta)\in C^1((\beta_0-\kappa,\beta_0]), u\neq0, u(\beta_0)<0,$ it holds that \begin{align}\label{eq311newadd1}u(\theta)<0,\ \theta\in(\beta_0-\kappa,\, \beta_0],\end{align} which, with $(\ref{Equ1})_2$, gives 
	\begin{equation}\label{AA1}
	w'(\theta)<0, \ \theta\in(\beta_0-\kappa,\,\beta_0].
	\end{equation} 
Then by $w(\beta_0)=\cos\beta_0>0$ from \eqref{Shockcondition}, we infer that 
\begin{align}\label{eq312new}w>0 ~\text{on}~ (\beta_0-\kappa,\,\beta_0].\end{align}
	
	2.~By Corollary \ref{Cor1} and Lemma \ref{Lemma111}, one checks that 
	\begin{equation}\label{AA2}
	p'(\theta)=(\frac{p}{\rho^{\epsilon+1}}\rho^{\epsilon+1})'(\theta)=(1+\epsilon)\frac{p}{\rho}\rho'(\theta)<0,\ \theta\in(\beta_0-\kappa,\,\beta_0].
	\end{equation} 
	
	3.~By $\eqref{A4}_2$, \eqref{eq243new} and \eqref{eq31new}, a direct computation yields 
	\begin{equation}\label{AA5}
c'(\theta)=\frac{\gamma(p(\theta)/\rho(\theta))'}{2c(\theta)}=\frac{-\epsilon(|V|^2(\theta))'}{2c(\theta)}<0,~\text{for}~\theta \in(\beta_0-\kappa,\,\beta_0].
	\end{equation} 
	
	4.~Also, by \eqref{eq243new}, \eqref{eq31new},  \eqref{eq311newadd1} and  \eqref{eq312new}, it follows  
	\begin{equation}\label{AA3}
	u'(\theta)<0,~\text{for}~ \theta\in(\beta_0-\kappa,\,\beta_0].
	\end{equation}
	
	5.~As $M_{n}\doteq {|u|}/{c}$ and $u(\theta)<0, \, \theta\in(\beta_0-\kappa,\beta_0],$  \eqref{AA3} and \eqref{AA5} imply that 
	\begin{equation}
M_n'(\theta)>0,~~~\text{for}~ \theta\in(\beta_0-\kappa,\,\beta_0]. 
	\end{equation} 
	Then by \eqref{AA6}, we infer 
	\begin{equation}\label{eq319new1}
0<M_n(\theta)<1,~\text{for}~ \theta\in(\beta_0-\kappa,\,\beta_0].
	\end{equation}
The proof is  complete.
\end{proof}

 \subsubsection{Solution of the inverse problem}
\begin{theorem}\label{t4}
	Let $\beta_*, \, \beta^*,\,  \bar{\beta}$ be constants with $0<\beta_*<\bar{\beta}<\beta^*<\frac{\pi}{2}.$ There exists an $\epsilon_*>0$ depending on  $\beta_*,\, \beta^*$, such that if $\epsilon\in(0,\epsilon_*),$  then the inverse problem with parameter $\epsilon$ and shock-front $\theta=\bar{\beta}$  is solvable. Furthermore, it holds that
	\begin{equation}\label{A24}
	\lim_{\epsilon\to 0}\bar{\theta}(\epsilon)=\bar{\beta},
	\end{equation}
	where $\theta=\bar{\theta}(\epsilon)$ is the cone's surface with respect to the fixed shock-front $\theta=\bar{\beta}$, and is  solved from the inverse problem corresponding to the parameter $\epsilon\in(0, \epsilon_*)$.
\end{theorem}
\begin{proof}
	1. We rewrite \eqref{Equ1} in the matrix form 
	\begin{equation}\label{A19}
	T\hat{U}'=B,
	\end{equation}
	 where 
	\begin{equation}
	T=\begin{pmatrix}
	u & \rho & 0 \\
	0 & 0 & 1 \\
	\gamma p/\rho^2 & u & 0
	\end{pmatrix},\quad \det T =c^2-u^2,	\end{equation}
	and 
	\begin{equation}
	\hat{U}=(\rho, \, u, \, w)^{\top},\, \ B=(\rho u\cot\theta+2\rho w,\, u, \, wu)^{\top}.
	\end{equation}
	The initial data $\hat{U}(\bar{\beta})$ is given by \eqref{Shockcondition}, with $\beta$ there replaced by $\bar{\beta}$, and depending on the parameter $\epsilon$. Notice that   $M_{0n}$ is also determined by $\epsilon$.   
Recall by the proof of Lemma \ref{Lemma111} that $\epsilon\in(0,\frac{\sin^2\beta_0}{E'})$ ensures  $M_{0n}>1.$
Then a direct computation yields
	\begin{equation}\label{A20old}
	\det T|_{\theta=\bar{\beta}} =c^2(\bar{\beta})-u^2(\bar{\beta})=c^2(\bar{\beta})\frac{(1+\epsilon/2)(M^2_{0n}-1)}{2(1+1/\epsilon)M^2_{0n}-1}>0.
	\end{equation}
Thus, assuming continuity of the solutions, there exists a positive constant $\kappa_1$, such that the equations  \eqref{A19} can be expressed  as follows:  
	\begin{equation}\label{A25}
	\hat{U}'=g(\theta,\hat{U}), \qquad \theta\in(\kappa_1, \bar{\beta}],
	\end{equation}
	where
	\begin{equation}
	g(\theta,\hat{U})=\begin{pmatrix}
	g_1(\theta,\hat{U}) \\
    g_2(\theta,\hat{U})\\
    g_3(\theta,\hat{U})
	\end{pmatrix}=\begin{pmatrix}
	\displaystyle{\frac{-\rho u^2\cot\theta-\rho wu}{c^2-u^2}} \\
	\displaystyle{\frac{c^2u\cot\theta+2c^2w-u^2w}{c^2-u^2}}\\
	u
	\end{pmatrix}.
	\end{equation}
One has 
\begin{align*}
&\frac{\partial g_1}{\partial \rho}=\displaystyle{\frac{(-u^2\cot\theta-wu)(c^2-u^2)+( u^2\cot\theta+ wu)\epsilon c^2}{(c^2-u^2)^2}},\\
&\frac{\partial g_1}{\partial u}=\displaystyle{\frac{(-2\rho u\cot\theta-\rho w)(c^2-u^2)-2u(\rho u^2\cot\theta+\rho wu)}{(c^2-u^2)^2}},\\
& \frac{\partial g_1}{\partial w}=\displaystyle{\frac{-\rho u}{c^2-u^2}},\\
&\frac{\partial g_2}{\partial \rho}=\displaystyle{\frac{\epsilon c^2\big((u\cot\theta+2w)(c^2-u^2)-c^2u\cot\theta-2c^2w+u^2w)\big)}{\rho~(c^2-u^2)^2}},\\
&\frac{\partial g_2}{\partial u}=\displaystyle{\frac{(c^2\cot\theta-2uw)(c^2-u^2)+2u(c^2u\cot\theta+2c^2w-u^2w)}{(c^2-u^2)^2}},\\
&\frac{\partial g_2}{\partial w}=\displaystyle{\frac{(2c^2-u^2)}{c^2-u^2}},\\
&\frac{\partial g_3}{\partial \rho}=0,\quad \frac{\partial g_3}{\partial u}=1,\quad \frac{\partial g_3}{\partial w}=0,
\end{align*}
Introduce the set 	
	\begin{equation}
	\begin{aligned}
	\Pi(\kappa_1)\doteq\Big\{(&\theta,\, \rho,\, u,\, w,\,c, \, M_n):\, \kappa_1<\theta<\bar{\beta}, \, \rho^{\gamma-1}(\bar{\beta})<\rho^{\gamma-1}<\frac{\epsilon E_0 \rho^{\gamma}(\bar{\beta})}{(\epsilon+1)p(\bar{\beta})},\\&u(\bar{\beta})<u<0, \,w(\bar{\beta})<w<2E_0, \,c(\bar{\beta})<c^2<\epsilon E_0,\,0<M_n<1\Big\},
	\end{aligned}
	\end{equation}	
where the upper bounds come from \eqref{eq250} and \eqref{eq319new1}. Then it holds that 
	\begin{equation}\label{A21}
	\sum_{j=1}^{3}\left(\norm{g_j}_{L^{\infty}(\Pi)}+\norm{\mathrm{D}_{\hat{U}}g_j}_{L^{\infty}(\Pi)}\right)<+\infty.
	\end{equation}
Hence, recalling that $\rho(\bar{\beta})>0$ and $u(\bar{\beta})<0$, by the local solvability of Cauchy problem of  ODE, as well as Lemma \ref{Lemma111} and Corollary \ref{Corollary111}, for  $\kappa_1$ close to $\bar{\beta}$,  we can solve a unique $C^1$ (actually analytical) solution $\hat{U}$ for  $\theta\in(\kappa_1,\bar{\beta})$, and $(\theta, \, \hat{U}(\theta),\,  c(\theta),\,  M_n(\theta))\in\Pi(\kappa_1)$.  Hence $\Pi(\kappa_1)$ is an invariant region for \eqref{A19}.  It then holds, as $(\theta,\,\rho,\,u,\,w,\,c,\,M_n)\in\Pi(\kappa_1),$ that 
	\begin{equation}\label{A20}
	\det T =c^2-u^2>c^2(\bar{\beta})-u^2(\bar{\beta})>0,
	\end{equation}
 so we could always use the expression \eqref{A25}  and extend the solution until one of the following holds:
 \begin{itemize}
 \item [i)] $\kappa_1=0$ and $u(\theta)<0$ for all $\theta\in(0, \bar{\beta}]$;
 
 \item [ii)] There is a $\bar{\theta}>0$ so that $\kappa_1=\bar{\theta}$ and  $u(\theta)<0$ for $\theta\in(\bar{\theta},\, \bar{\beta}]$, and  $u(\bar{\theta})=0.$
 \end{itemize}
(It is clear that from the restriction $E\equiv E_0$, the lower bound of  $w$ rejects that $\rho$ and $c^2$ reach their upper bound in $\Pi(\kappa_1)$, while the lower bound of $\rho$  eliminates the possibility that $w$ achieves its upper bound. So thanks to \eqref{eq319new1}, the above i) and ii) are the only options for a solution to escape from the invariant region.)	
	
2.~We show that the  case i) won't happen; in other words,  case ii) happens: there is a  $\bar{\theta}\in(0,\bar{\beta})$ so that   $u(\bar{\theta})=0$.	
	
	Indeed,  suppose that for any $\bar{\theta}>0$, it holds that $u(\bar{\theta})<0$. From the third equation in \eqref{Equ1} and Corollary \ref{Corollary111},  we have 
\begin{equation}\label{eq326new}
u'(\theta)=-\frac{p'}{\rho u}(\theta)-w(\theta)<-w(\theta)<-w(\bar{\beta})=-\cos\bar{\beta},\quad \theta\in(\bar{\theta}, \bar{\beta}),
\end{equation}
thus 	\begin{equation}\label{A22}
    \begin{aligned}
    	u(\bar{\beta})&\le u(\bar{\beta})-u(\bar{\theta})=\int_{\bar{\theta}}^{\bar{\beta}}u'(\theta)\,\mathrm{d}\theta\\&\leq \int_{\bar{\theta}}^{\bar{\beta}}(-\cos\bar{\beta})d\theta=-(\bar{\beta}-\bar{\theta})\cos\bar{\beta},
    \end{aligned}
	\end{equation} 
	which provides 
	\begin{align}\label{A23}	|\bar{\beta}-\bar{\theta}|&\le \frac{|u(\bar{\beta})|}{\cos\bar{\beta}}=\frac{2E'+\sin^2\bar{\beta}}{2E'M_0^2\sin^2\bar{\beta}+1}\cdot\tan\bar{\beta}\nonumber\\&=\epsilon\frac{2E'+\sin^2\bar{\beta}}{2\sin^2\bar{\beta}+\epsilon}\cdot\tan\bar{\beta}\nonumber\\&\leq\epsilon\frac{2E'+\sin^2{\beta^*}}{2\sin^2{\beta_*}}\cdot\tan{\beta^*}.
	\end{align}
Here, for the first equality, we used  the fact that 
\begin{equation*}
u(\bar{\beta})=\frac{2+\epsilon M_{0n}^2}{(\epsilon+2)M_{0n}^2}\cdot u_0=-\frac{2E'+\sin^2\bar{\beta}}{2E'M_0^2\sin^2\bar{\beta}+1}\cdot\sin\bar{\beta},
\end{equation*}
which follows from \eqref{Shockcondition}; and for the second equality, we used \eqref{2.7}.

Therefore,  as a result of \eqref{A23}, taking $$\epsilon_*=\min\Big\{\frac{\sin^2\beta_*}{E'}, \frac{\beta_*\sin^2\beta_*}{(2E'+\sin^2\beta^*)\tan\beta^*}\Big\},$$ then if $\epsilon\in(0,\epsilon_*),$ one has 
$$|\bar{\beta}-\bar{\theta}|<\frac12\beta_*.$$
Since this holds for any $\bar{\theta}>0$ for case i), it follows that
$$\bar{\beta}\le\frac12\beta_*,$$
which contradicts to the requirement that $\beta_*<\bar{\beta}$. 
   
Hence we have shown that there exists a $\bar{\theta}$ that depends on $\epsilon>0$, and the $C^1$ flow field in $(\bar{\theta},\, \bar{\beta})$ that solves the inverse problem.

3. Furthermore, for case ii), noticing that $u(\bar{\theta})=0$,   \eqref{A24} follows directly from \eqref{A23}. The proof is complete.	
\end{proof}

\subsection{Solvability of the direct problem}

Based on Theorem \ref{t4}, for given $0<\epsilon<\epsilon_*$, we define the following  mapping from the conical shock-front to the cone's surface
\begin{equation}
\begin{aligned}
\mathcal{T}: ~(\beta_*,\,\beta^*) \longrightarrow (0,\, \beta^*),~ \bar{\beta}\longmapsto\bar{\theta}.
\end{aligned}
\end{equation}
It has been shown that $\mathcal{T}(\bar{\beta})<\bar{\beta}$.

\begin{proposition}\label{Proposition1}
	The mapping $\mathcal{T}$ is continuous.
\end{proposition}
\begin{proof}
	For fixed $\beta_0\in (\beta_*,\, \beta^*)$ and any sequence $(\max\{\beta_*,\, \mathcal{T}(\beta_0)\},\, \beta^*)\ni\beta_n\to\beta_0,$ $(n\to\infty)$,  we classify $\beta_n$ into two subsequences $$\tilde{\beta}_{n}\in\mathcal{B}_1\doteq\{\beta_n:\mathcal{T}(\beta_n)\leq\mathcal{T}(\beta_0)\}, ~\quad~\tilde{\tilde{\beta}}_n\in\mathcal{B}_2\doteq\{\beta_n:\mathcal{T}(\beta_n)>\mathcal{T}(\beta_0)\}.$$ Let $U(\theta;\beta)$ be the solution of the inverse problem with the given shock position $\beta$. Then we have
	\begin{equation}
	\begin{aligned}
	0&=u(\mathcal{T}(\beta_0);\beta_0)-u(\mathcal{T}(\tilde{\beta}_n);\tilde{\beta}_n)\\&=\{u(\mathcal{T}(\beta_0);\beta_0)-u(\mathcal{T}(\beta_0);\tilde{\beta}_n)\}+\{u(\mathcal{T}(\beta_0);\tilde{\beta}_n)-u(\mathcal{T}(\tilde{\beta}_n);\tilde{\beta}_n)\}\\&=\{u(\mathcal{T}(\beta_0);\beta_0)-u(\mathcal{T}(\beta_0);\beta_n)\}\\
	&\quad+\int_{0}^{1}u'(\tau\mathcal{T}(\beta_0)+(1-\tau)\mathcal{T}(\tilde{\beta}_n);\tilde{\beta}_n)\,\mathrm{d}\tau\,(\mathcal{T}(\beta_0)-\mathcal{T}(\tilde{\beta}_n)),
	\end{aligned}
	\end{equation}
	which gives that
	\begin{equation}\label{Continousdependence 1}
	\begin{aligned}
	|\mathcal{T}(\beta_0)-\mathcal{T}(\tilde{\beta}_n)|\leq\frac{|u(\mathcal{T}(\beta_0);\beta_0)-u(\mathcal{T}(\beta_0);\tilde{\beta}_n)|}{\left|\int_{0}^{1}u'(\tau\mathcal{T}(\beta_0)+(1-\tau)\mathcal{T}(\tilde{\beta}_n);\tilde{\beta}_n)\,\mathrm{d}\tau\right|}.
	\end{aligned}
	\end{equation}
Then, by the continuous dependence of solutions on initial data for Cauchy problem of ODE \eqref{A25},  and the fact that $u'$ is uniformly bounded away from zero (cf. \eqref{eq326new}), it follows from \eqref{Continousdependence 1} that
	\begin{equation}\label{CD1}
	\lim_{\tilde{\beta}_n\to\beta_0}\mathcal{T}(\tilde{\beta}_n)=\mathcal{T}(\beta_0).
	\end{equation}
	Similar argument also yields 
	\begin{equation}\label{CD2}
	\lim_{\tilde{\tilde{\beta}}_n\to\beta_0}\mathcal{T}(\tilde{\tilde{\beta}}_n)=\mathcal{T}(\beta_0).
	\end{equation}
	Thus the map $\mathcal{T}$ is continuous.
	\end{proof}

Based on Theorem \ref{t4} and Proposition \ref{Proposition1}, we establish the solvability of direct problem as follows.

\begin{theorem}\label{T1}
	For any given cone's surface $\theta=\theta_0$ with  $\theta_0\in(0,\,\frac{\pi}{2}),$ there exists an $\epsilon^*>0$ such that  the direct problem is solvable for $\epsilon\in(0,\,\epsilon^*)$. That is, for each fixed $\epsilon$, there exists a $\beta_0$ such that $\mathcal{T}(\beta_0)=\theta_0.$ Furthermore, it holds that 
	\begin{equation}\label{A26}
	\lim_{\epsilon\to0}|\beta_0-\theta_0|=0.
	\end{equation}

\end{theorem}
\begin{proof}
	Let $\theta_1$ be a constant in $(\theta_0,\, \frac{\pi}{2}).$ By \eqref{A24} in Theorem \ref{t4}, there exists a positive constant $\epsilon_1$ such that for $\epsilon\in(0,\, \epsilon_1),$ $\mathcal{T}(\theta_1)\in(\theta_0,\theta_1),$ namely $\theta_0\in(\mathcal{T}(\theta_0),\, \mathcal{T}(\theta_1)).$ By the continuity of $\mathcal{T},$ there is a $\beta_0\in(\theta_0,\, \theta_1)$ such that $\mathcal{T}(\beta_0)=\theta_0.$ 
	
	Furthermore, by \eqref{A23} in the proof of Theorem \ref{t4}, one has
	\begin{equation}
	|\beta_0-\theta_0|\leq\epsilon\frac{2E'+\sin^2{\theta_1}}{2\sin^2{\theta_0}}\cdot\tan{\theta_1},
	\end{equation}
	which implies \eqref{A26}. The proof is complete.
\end{proof}

As a corollary of  Theorem \ref{T1}, we have the following. 
\begin{corollary}
		For any given cone's surface $\theta=\theta_0\in(0,\,\frac{\pi}{2}),$ there exists an $\epsilon^*>0$ such that \textsc{Problem 1} has a  weak entropy solution for any $\epsilon\in(0,\,\epsilon^*)$. Indeed, $U$ is given by \eqref{ss} with $U|_{\theta_0<\theta<\beta_0}$ determined by Theorem \ref{T1}. 
\end{corollary}

\section{Measure solutions and the hypersonic limit}\label{sec4}
\subsection{Measure solutions for hypersonic-limit flow}

Theorem \ref{T1} establishes the solvability of  the direct problem and the convergence of the corresponding shock-front to the boundary surface of the cone as Mach number of the incoming flow goes to infinity, i.e., $\epsilon$ goes to zero.  In this section, we further clarify the convergence of the flow states $U$ to a singular Radon measure solution  with concentration of mass and momentum in an infinite-thin boundary layer, and demonstrate the Newtonian sine-squared pressure law for cones in  hypersonic aerodynamics. 

Referring to \cite[Theorem 4.1]{QY}, we directly present the Radon measure solution for the singular case (i.e., $\epsilon=0$). 

\begin{theorem}[hypersonic-limit flows passing cones] \label{t4.1}
    For hypersonic limit flows passing the cone with surface $\theta=\theta_0\in(0,\, \frac{\pi}{2}),$  (i.e., $\epsilon=0$ in  \textsc{Problem $1'$}),   there is a unique \footnote{The uniqueness is shown in the class of Radon measure solutions with the structure exhibited by \eqref{49} and \eqref{4.10}, namely the weights of the Dirac measures are unique.} Radon measure solution (in the sense of  Definition \ref{Def2}) $\mathcal{U}=(\varrho,\, \tilde{u},\, w,\, E)$ given by 
     \begin{align}
    &\varrho=\rho_0\mathrm{d}A\mres{\Omega}+\frac{1}{2}\tan\theta_0\delta_C,~\quad \tilde{u}=\tilde{u}_0I_\Omega, \label{49}\\
    & w=w_0I_\Omega+\cos\theta_0I_C,~\quad E=E_0I_\Omega+E_0I_C, \label{4.10}
    \end{align}
in which $(\rho_0,\tilde{u}_0, w_0, E_0)$ is given by \eqref{upcoming}, $C=\{(r,\theta,\phi):\, r=1,\, \theta=\theta_0\}$, and $I_{C}(\cdot)$ is the indicator function of the set $C$.  
Furthermore, we have
\begin{equation}
    W_C=\sin^2\theta_0, \label{4.11}
\end{equation}
which is  the Newtonian sine-squared pressure law.
\end{theorem}

Now for given $\theta_0\in(0,\pi/2)$ and any $\epsilon\in(0,\epsilon^*) $ as given in Theorem \ref{T1}, we write the solved $\beta_0$ to be $\beta^\epsilon$,  
and denote $\Omega_1^\epsilon$ as the region (on the sphere $S^2$) ahead of the shock-front, namely $\theta>\beta^\epsilon$, and set $\Omega_2^\epsilon\doteq\Omega\setminus\Omega_1^\epsilon$ to be the region bounded by the shock-front and the cone on $S^2$.
Let $U^{\epsilon}=(\rho^{\epsilon},\, \tilde{u}^{\epsilon},\, w^{\epsilon},\, E^{\epsilon})$ be the weak entropy solution to \textsc{Problem $1'$} established in Theorem \ref{T1}. Then we have the following Radon measures on momentum, energy flux, momentum fluxes, mass, and pressure: 
\begin{align}
    & m_a^\epsilon=\rho_0\tilde{u}_0\mathrm{d}A\mres\Omega^{\epsilon}_1+\rho^\epsilon \tilde{u}^{\epsilon} \mathrm{d}A\mres\Omega^{\epsilon}_2,~ ~~~~~~~~~~ n_a^\epsilon=\rho_0w_0\mathrm{d}A\mres\Omega^{\epsilon}_1+\rho^\epsilon w^\epsilon \mathrm{d}A\mres\Omega^{\epsilon}_2 \label{4.1.35}\\
    & m_e^\epsilon=\rho_0\tilde{u}_0E_0\mathrm{d}A\mres\Omega^{\epsilon}_1+\rho^\epsilon E_0\tilde{u}^{\epsilon} \mathrm{d}A\mres\Omega^{\epsilon}_2,~ ~~~~ n_e^\epsilon=\rho_0w_0E_0\mathrm{d}A\mres\Omega^{\epsilon}_1+\rho^\epsilon \tilde{u}^{\epsilon} E_0\mathrm{d}A\mres\Omega^{\epsilon}_2,~ \label{4.1.36}\\
    & m_r^\epsilon=\rho_0\tilde{u}_0w_0\mathrm{d}A\mres\Omega^{\epsilon}_1+\rho^\epsilon \tilde{u}^{\epsilon} w^\epsilon \mathrm{d}A\mres\Omega^{\epsilon}_2,~ ~~~~ n_r^\epsilon=\rho_0w_0^2\mathrm{d}A\mres\Omega^{\epsilon}_1+\rho^\epsilon (w^\epsilon )^2\mathrm{d}A\mres\Omega^{\epsilon}_2,~\label{4.1.37}\\
    & m_t^\epsilon=\rho_0\tilde{u}_0\otimes \tilde{u}_0\mathrm{d}A\mres\Omega^{\epsilon}_1+\rho^\epsilon \tilde{u}^{\epsilon}\otimes \tilde{u}^{\epsilon} \mathrm{d}A\mres\Omega^{\epsilon}_2,~ \nonumber\\
    &n_t^\epsilon=\rho_0|\tilde{u}_0|^2\mathrm{d}A\mres\Omega^{\epsilon}_1+\rho^\epsilon |\tilde{u}^{\epsilon}|^2\mathrm{d}A\mres\Omega^{\epsilon}_2,~\label{4.1.38}\\
    & \varrho^\epsilon=\rho_0\mathrm{d}A\mres\Omega^{\epsilon}_1+\rho^\epsilon \mathrm{d}A\mres\Omega^{\epsilon}_2,~  ~~~~~~~~~~~~~~~~~~~~~~\wp^\epsilon=p_0\mathrm{d}A\mres\Omega^{\epsilon}_1+p^\epsilon \mathrm{d}A\mres\Omega^{\epsilon}_2.\label{4.1.39}
\end{align}
It is easily seen that $\mathcal{U}^{\epsilon}\doteq(\varrho^\epsilon, u^{\epsilon}, w^{\epsilon}, E^{\epsilon})$ is a Radon measure solution. For the singular case $\epsilon=0$, due to Theorem \ref{t4.1}, the corresponding measures turn out to be
\begin{align}
    & m_a=\rho_0\tilde{u}_0\mathrm{d}A\mres\Omega,\quad n_a=\rho_0w_0\mathrm{d}A\mres\Omega+\frac{1}{2}\sin\theta_0\delta_C,~ \label{4.1.30}\\
    & m_e=\rho_0\tilde{u}_0E_0\mathrm{d}A\mres\Omega,\quad  n_e=\rho_0w_0E_0\mathrm{d}A\mres\Omega+\frac{1}{2}\sin\theta_0E_0\delta_C,~ \label{4.1.31}\\
    & m_r=\rho_0\tilde{u}_0w_0\mathrm{d}A\mres\Omega,\quad  n_r=\rho_0w_0^2\mathrm{d}A\mres\Omega+\frac{1}{2}\sin\theta_0\cos\theta_0\delta_C,~\label{4.1.32}\\
    & m_t=\rho_0\tilde{u}_0\otimes \tilde{u}_0\mathrm{d}A\mres\Omega\quad n_t=\rho_0|\tilde{u}_0|^2\mathrm{d}A\mres\Omega,~\label{4.1.33}\\
    & \varrho=\rho_0\mathrm{d}A\mres\Omega+\frac{1}{2}\tan\theta_0\delta_C,\quad \wp=0.\label{4.1.34}
\end{align}

\begin{lemma}\label{lem41}
The measures defined in \eqref{4.1.35}-\eqref{4.1.39}  converge vaguely to the corresponding measures defined in \eqref{4.1.30}-\eqref{4.1.34} as $\epsilon\to0$. Furthermore, it holds that 
\begin{equation}
\lim_{\epsilon\to0}p^{\epsilon}(\theta_0)=W_C.
\end{equation}	
\end{lemma}
\begin{proof}
	We  present detailed  proof on $m^{\epsilon}_a\to m_a, \, n^{\epsilon}_a\to n_a, \, n^{\epsilon}_r\to n_r$ and $\varrho^{\epsilon}\to \varrho.$ The others can be proved similarly. 
	
	1.~According to Definition \ref{Def2} and the fact that $v_0=0$ (cf. \eqref{upcoming}), it holds that, for any $C^1$ vector field $\Psi$ on $S^2$, 
	\begin{align}
	& \langle m_a^\epsilon,\, \Psi\rangle=\int_{\Omega_1^\epsilon}\left(\Psi(\theta,\phi),\rho_0u_0\right)\,\mathrm{d}A+\int_{\Omega_2^\epsilon}\left(\Psi(\theta,\phi),\rho^\epsilon(\theta)u^\epsilon(\theta)\right)\,\mathrm{d}A,~\\
	& \langle m_a,\, \Psi\rangle=\int_{\Omega}\left(\Psi(\theta,\, \phi),\rho_0u_0\right)\,\mathrm{d}A.
	\end{align}
A direct computation shows that 
	\begin{equation}\label{A29}
	\begin{aligned}
	&\quad \quad~ \left|\langle m_a^\epsilon,\,\Psi\rangle- \langle m_a,\,\Psi\rangle\right|=\left|  \int_{\Omega_2^\epsilon}\left(\Psi (\theta,\phi),\big(\rho^\epsilon(\theta)u^\epsilon(\theta)-\rho_0u_0\big)\right)\,\mathrm{d}A\right|\\&
	\leq 
\norm{\Psi}_{L^\infty(S^2)}\norm{\frac{|u^\epsilon|}{w^\epsilon}}_{L^\infty(\Omega_2^{\epsilon})}\int_{\Omega_2^\epsilon}\rho^\epsilon(\theta)w^\epsilon(\theta)\,\mathrm{d}A+\norm{\Psi}_{L^\infty(S^2)}\norm{\rho_0u_0}_{L^\infty(\Omega)}\Big|\Omega_2^{\epsilon}\Big|.
	\end{aligned}
	\end{equation}
Here and in the sequel, $\Big|\Omega_2^{\epsilon}\Big|\doteq\mathrm{d}A\left(\Omega_2^{\epsilon}\right)$, with $\mathrm{d}A$ being the standard area measure on $S^2$.

We now estimate the right-hand-side of \eqref{A29} term by term. The estimate \eqref{A26} (where $\beta_0$ is replaced by $\beta^\epsilon$ now) implies that the area $|\Omega_2^\epsilon|\to0$ as $\epsilon\to0$. Thus the second  term in the right side of \eqref{A29} goes to $0$. For the first term, 
notice that by the conservation of mass, namely \eqref{eq1},  it follows directly the important identity  
\begin{equation}\label{**}
	\int_{\Omega_2^\epsilon}\rho^\epsilon(\theta)w^\epsilon(\theta)~\mathrm{d}A=2\pi\int_{\theta_0}^{\beta^{\epsilon}}\rho^\epsilon(\theta)w^\epsilon(\theta)\sin\theta~\mathrm{d}\theta=\pi \sin^2\beta^\epsilon. 
	\end{equation}
According to Corollary \ref{Corollary111} and \eqref{Shockcondition}, and recalling  $\beta^\epsilon\to\theta_0>0,$ it holds that 
	 \begin{equation}\label{A28}
	 \left|\frac{u^\epsilon(\theta)}{w^\epsilon(\theta)}\right|<\left|\frac{u^\epsilon(\beta^{\epsilon})}{w^\epsilon(\beta^{\epsilon})}\right|=\tan\beta^{\epsilon}\left(\frac{2E}{\sin^2\beta^{\epsilon}}+1\right)\frac{\epsilon}{\epsilon+2}\to 0,\quad (\epsilon\to0).	 \end{equation}
 Therefore, we have, for any vector field $\Psi\in C^1(S^2),$
	\begin{equation}\label{A27}
	\lim_{\epsilon\to 0}\Big|\langle m_a^\epsilon,\,\Psi\rangle- \langle m_a,\,\Psi\rangle\Big|=0,
	\end{equation}
	that is, $m^{\epsilon}_a\to m_a$ vaguely. For the convergence of measures $m_e^\epsilon,~m_r^\epsilon,~m_t^\epsilon,~n_t^\epsilon,$ the proof remains the same as  $m^{\epsilon}_a.$
	
	2.~Next we prove  $n_a^\epsilon \to n_a$.  For any $\psi\in C(S^2)$,  one has
	\begin{align}
	 \langle n_a^\epsilon,\, \psi\rangle&=\int_{\Omega_1^\epsilon}\psi(\theta,\phi)\rho_0w_0\,\mathrm{d}A+\int_{\Omega_2^\epsilon}\psi(\theta,\phi)\rho^\epsilon(\theta)w^\epsilon(\theta)\,\mathrm{d}A,~\\
	 \langle n_a,\, \psi\rangle&=\int_{\Omega}\psi(\theta,\, \phi)\rho_0w_0\,\mathrm{d}A+\int_{C}\psi(\theta,\, \phi)\frac{1}{2}\sin\theta_0\,\mathrm{d}s\nonumber\\&=\int_{\Omega}\psi(\theta,\, \phi)\rho_0w_0\,\mathrm{d}A+\int_{0}^{2\pi}\psi(\theta_0,\, \phi)\frac{1}{2}\sin^2\theta_0\,\mathrm{d}\phi.
	\end{align}
By \eqref{**}, it follows that 
	\begin{align}\label{4.18}	
	&\quad \quad~\Big|\langle n_a^\epsilon,\,\psi\rangle-\langle n_a,\,\psi\rangle\Big| \nonumber \\&
	=\Big|\int_{\Omega_2^\epsilon}\psi(\theta,\phi)\rho^\epsilon(\theta)w^\epsilon(\theta)\,\mathrm{d}A-\int_{\Omega_2^\epsilon}\psi(\theta,\phi)\rho_0w_0\,\mathrm{d}A-\int_0^{2\pi}\psi(\theta_0,~\phi)\frac{1}{2}\sin^2\theta_0\, \mathrm{d}\phi\Big|\nonumber\\ &
	\leq \Big|\int_0^{2\pi}\int_{\theta_0}^{\beta^\epsilon} [\psi(\theta,~\phi)-\psi(\theta_0,~\phi)]\rho^\epsilon w^\epsilon \sin\theta\,\mathrm{d}\theta\,\mathrm{d}\phi\Big|+\norm{\psi}_{L^{\infty}(S^2)}\norm{\rho_0w_0}_{L^{\infty}(\Omega)}\Big|\Omega_2^{\epsilon}\Big|
	\nonumber\\&\quad\quad~+ \Big|\int_0^{2\pi}\int_{\theta_0}^{\beta^\epsilon} \psi(\theta_0,~\phi)\rho^\epsilon(\theta)w^\epsilon(\theta)\sin\theta\,\mathrm{d}\theta\mathrm{d}\phi-\int_0^{2\pi}\psi(\theta_0,~\phi)\frac{1}{2}\sin^2\theta_0\,\mathrm{d}\phi\Big|\nonumber\\ & \leq \frac{1}{2}\sin^2\beta^{\epsilon} \int_0^{2\pi}\sup_{\theta\in[\theta_0,~\beta^\epsilon]}\Big|\psi(\theta,~\phi)-\psi(\theta_0,~\phi)\Big|\,\mathrm{d}\phi+\norm{\psi}_{L^{\infty}(S^2)}\norm{\rho_0w_0}_{L^{\infty}(\Omega)}\Big|\Omega_2^{\epsilon}\Big|\nonumber\\&\quad\quad~+ \pi\norm{\psi}_{L^{\infty}(S^2)}\Big|\sin^2\beta^{\epsilon}-\sin^2\theta_0\Big|,
	\end{align}
which, together with $\lim_{\epsilon\to0}|\Omega_2^{\epsilon}|=0$ and  $\lim_{\epsilon\to0}\beta^{\epsilon}=\theta_0$, gives that 
\begin{equation}
\lim_{\epsilon\to0} \langle n_a^\epsilon,\, \psi\rangle=\langle n_a,\, \psi\rangle;
\end{equation}
that is, $n_a^\epsilon$ converges to $n_a$ vaguely. And then, by a similar process shown in \eqref{4.18}, we can prove  $n_e^\epsilon\to n_e$ vaguely for $\epsilon\to0$. 
	
	3.~We next prove that $n_r^\epsilon\to n_r$ vaguely for $\epsilon\to0.$ For any $\psi\in C(S^2),$ one has 
	\begin{align}
	\langle n_r^\epsilon,\, \psi\rangle&=\int_{\Omega_1^\epsilon}\psi(\theta,\phi)\rho_0w^2_0\,\mathrm{d}A+\int_{\Omega_2^\epsilon}\psi(\theta,\phi)\rho^\epsilon(\theta)(w^\epsilon)^2(\theta)\,\mathrm{d}A,~\\
	 \langle n_r,\, \psi\rangle&=\int_{\Omega}\psi(\theta,\, \phi)\rho_0w^2_0\,\mathrm{d}A+\int_{C}\psi(\theta,\, \phi)\frac{1}{2}\sin\theta_0\cos\theta_0\,\mathrm{d}s\nonumber\\&=\int_{\Omega}\psi(\theta,\, \phi)\rho_0w^2_0\,\mathrm{d}A+\int_{0}^{2\pi}\psi(\theta_0,\, \phi)\frac{1}{2}\sin^2\theta_0\cos\theta_0\,\mathrm{d}\phi.
	\end{align}
Then noticing that $\rho(\theta)>0, w^{\epsilon}(\theta)>0$ for $\theta\in(\theta_0,\beta^{\epsilon})$,  $w(\beta^{\epsilon})=\cos\beta^{\epsilon}$ (by $\eqref{Shockcondition}_6$ with $\beta$ replaced by $\beta^{\epsilon}$), and \eqref{**}, we have  
	\begin{align}
	&\quad \quad~\Big|\langle n_r^\epsilon,\, \psi\rangle-\langle n_r,\, \psi\rangle\Big|\nonumber\\&=\Big|\int_{\Omega_2^{\epsilon}}\psi(\theta,\phi)\rho^{\epsilon}(\theta)(w^{\epsilon})^2(\theta)~\mathrm{d}A-\int_{\Omega_2^{\epsilon}}\psi(\theta,\phi)\rho_0w_0^2~\mathrm{d}A-\int_{0}^{2\pi}\psi(\theta_0,\, \phi)\frac{1}{2}\sin^2\theta_0\cos\theta_0\,\mathrm{d}\phi\Big|\nonumber\\&\leq\Big|\int_{\Omega_2^\epsilon}\psi(\theta,\phi)\rho^\epsilon w^\epsilon (w^\epsilon(\theta)-w^\epsilon(\beta^\epsilon))\mathrm{d}A+w^\epsilon(\beta^\epsilon)\int_{\Omega_2^\epsilon}\psi(\theta,\phi)\rho^\epsilon w^\epsilon \mathrm{d}A\nonumber\\&\quad\quad~-\int_{0}^{2\pi}\psi(\theta_0,\, \phi)\frac{1}{2}\sin^2\theta_0\cos\theta_0\,\mathrm{d}\phi\Big|+\Big|\int_{\Omega_2^{\epsilon}}\psi(\theta,\phi)\rho_0w_0^2~\mathrm{d}A\Big|\nonumber\\&\leq\Big|w^\epsilon(\beta^\epsilon)\int_{\Omega_2^\epsilon}\psi(\theta,\phi)\rho^\epsilon w^\epsilon \mathrm{d}A-\int_{0}^{2\pi}\psi(\theta_0,\, \phi)\frac{1}{2}\sin^2\theta_0\cos\theta_0\,\mathrm{d}\phi\Big|\nonumber\\&\quad\quad~+\Big|\int_{\Omega_2^\epsilon}\psi(\theta,\phi)\rho^\epsilon w^\epsilon (w^\epsilon(\theta)-w^\epsilon(\beta^\epsilon))\mathrm{d}A\Big|+\Big|\int_{\Omega_2^{\epsilon}}\psi(\theta,\phi)\rho_0w_0^2~\mathrm{d}A\Big|\nonumber\\&\leq\Big|w^\epsilon(\beta^\epsilon)\int_{\Omega_2^\epsilon}\psi(\theta_0,\phi)\rho^\epsilon w^\epsilon \mathrm{d}A-\int_{0}^{2\pi}\psi(\theta_0,\, \phi)\frac{1}{2}\sin^2\theta_0\cos\theta_0\,\mathrm{d}\phi\Big|\nonumber\\&\quad\quad~+\Big|w^\epsilon(\beta^\epsilon)\int_{\Omega_2^\epsilon}(\psi(\theta,\phi)-\psi(\theta_0,\phi))\rho^\epsilon w^\epsilon \mathrm{d}A\Big|\nonumber\\&\quad\quad~+\Big|\int_{\Omega_2^\epsilon}\psi(\theta,\phi)\rho^\epsilon w^\epsilon (w^\epsilon(\theta)-w^\epsilon(\beta^\epsilon))\mathrm{d}A\Big|+\Big|\int_{\Omega_2^{\epsilon}}\psi(\theta,\phi)\rho_0w_0^2~\mathrm{d}A\Big|\nonumber\\&\leq \pi\norm{\psi}_{L^{\infty}(S^2)}\Big|\sin^2\beta^{\epsilon}\cos\beta^{\epsilon}-\sin^2\theta_0\cos\theta_0\Big|+\norm{\psi}_{L^{\infty}(S^2)}\norm{\rho w_0^2}_{L^{\infty}(\Omega)}\Big|\Omega^{\epsilon}_2\Big|\nonumber\\&\quad\quad~+\pi\norm{\psi}_{L^{\infty}(S^2)}\sup_{\theta\in[\theta_0,~\beta^\epsilon]}|w^\epsilon(\theta)-w^\epsilon(\beta^\epsilon)|\sin^2\beta^{\epsilon}\nonumber\\&\quad\quad~+\frac{1}{2}\sin^2\beta^{\epsilon}\cos\beta^{\epsilon}\int_{0}^{2\pi}\sup_{\theta\in[\theta_0,~\beta^\epsilon]}|\psi(\theta,\phi)-\psi(\theta_0,\phi)|~\mathrm{d}\phi.
	\end{align}
	Thanks to  $\lim_{\epsilon\to0}|\Omega_2^{\epsilon}|=0$ and $\lim_{\epsilon\to0}\beta^{\epsilon}=\theta_0$, one has 
	\begin{equation}	
	\lim_{\epsilon\to0} \langle n_r^\epsilon,\, \psi\rangle=\langle n_r,\, \psi\rangle,
	\end{equation}
	hence  $n_r^\epsilon$ converges to $n_r$ vaguely.

4.~Now we show that $\varrho^{\epsilon}$ converges to $\varrho$ vaguely for $\epsilon\to0$. For any $\psi\in C(S^2),$ one has 
\begin{align}
\langle\varrho^\epsilon,~\psi\rangle&=\int_{\Omega_1^{\epsilon}}\psi(\theta,~\phi)\rho_0~\mathrm{d}A+\int_{\Omega_2^\epsilon}\psi(\theta,~\phi)\rho^\epsilon(\theta)~\mathrm{d}A,~\\
\langle\varrho,~\psi\rangle&=\int_{\Omega}\psi(\theta,~\phi)\rho_0~\mathrm{d}A+\int_C\psi(\theta,~\phi)\frac{1}{2}\tan\theta_0~\mathrm{d}s\nonumber\\&=\int_{\Omega}\psi(\theta,~\phi)\rho_0~\mathrm{d}A+\int_{0}^{2\pi}\psi(\theta_0,\phi)\frac{1}{2}\tan\theta_0\sin\theta_0\,\mathrm{d}\phi.
\end{align}
Thanks to the facts that $w(\beta^\epsilon)=\cos\beta^\epsilon, ~\rho^{\epsilon}=\rho^{\epsilon}w^{\epsilon}\Big(\frac{1}{w^\epsilon}-\frac{1}{w^\epsilon(\beta^\epsilon)}+\frac{1}{w^\epsilon(\beta^\epsilon)}\Big)$, and \eqref{**}, one has 
\begin{align}
&\quad\quad~\Big|\langle\varrho^\epsilon,~\psi\rangle-\langle\varrho,~\psi\rangle\Big|\nonumber\\&\leq\Big|\int_{\Omega_2^\epsilon}\psi(\theta,~\phi)\rho^\epsilon(\theta)~\mathrm{d}A-\int_{0}^{2\pi}\psi(\theta_0,\phi)\frac{1}{2}\tan\theta_0\sin\theta_0\,\mathrm{d}\phi\Big|+\Big|\int_{\Omega^{\epsilon}_2}\psi(\theta,~\phi)\rho_0~\mathrm{d}A\Big|\nonumber\\
&\leq\Big|\int_{\Omega_2^\epsilon}\psi(\theta,\phi)\rho^\epsilon w^\epsilon \frac{1}{w^\epsilon(\beta^\epsilon)}~\mathrm{d}A-\int_{0}^{2\pi}\psi(\theta_0,\phi)\frac{1}{2}\tan\theta_0\sin\theta_0\,\mathrm{d}\phi\Big|\nonumber\\
&\quad\quad~+\Big|\int_{\Omega_2^\epsilon} \psi(\theta,\phi)\rho^\epsilon w^\epsilon \Big(\frac{1}{w^\epsilon(\theta)}-\frac{1}{w^\epsilon(\beta^\epsilon)}\Big)~\mathrm{d}A\Big|+\norm{\psi}_{L^{\infty}(S^2)}\norm{\rho_0}_{L^{\infty}(\Omega)}\Big|\Omega^{\epsilon}_2\Big|\nonumber
\\&\leq\Big|\int_{\Omega_2^\epsilon}\psi(\theta_0,\phi)\rho^\epsilon w^\epsilon \frac{1}{w^\epsilon(\beta^\epsilon)}~\mathrm{d}A-\int_{0}^{2\pi}\psi(\theta_0,\phi)\frac{1}{2}\tan\theta_0\sin\theta_0\,\mathrm{d}\phi\Big|\nonumber\\
&\qquad\qquad+\norm{\psi}_{L^{\infty}(S^2)}\norm{\rho_0}_{L^{\infty}(\Omega)}\Big|\Omega^{\epsilon}_2\Big|+\int_{\Omega_2^\epsilon}|\psi(\theta,\phi)-\psi(\theta_0,\phi)|\rho^\epsilon w^\epsilon \frac{1}{w^\epsilon(\beta^\epsilon)}~\mathrm{d}A\nonumber\\
&\qquad\qquad\quad\quad~+
\pi\sin^2\beta^{\epsilon}\norm{\psi}_{L^{\infty}(S^2)} \sup_{\theta\in[\theta_0,\beta^{\epsilon}]}\Big|\frac{1}{w^\epsilon(\theta)}-\frac{1}{w^\epsilon(\beta^\epsilon)}\Big|\nonumber
\\&\leq\pi\norm{\psi}_{L^{\infty}(S^2)}\Big|\tan\beta^{\epsilon}\sin\beta^{\epsilon}-\tan\theta_0\sin\theta_0\Big|+\norm{\psi}_{L^{\infty}(S^2)}\norm{\rho_0}_{L^{\infty}(\Omega)}\Big|\Omega^{\epsilon}_2\Big|\nonumber\\
&\quad\quad~+ \frac{1}{2}\tan\beta^{\epsilon}\sin\beta^{\epsilon}\int_{0}^{2\pi}\sup_{\theta\in[\theta_0,~\beta^\epsilon]}|\psi(\theta,\phi)-\psi(\theta_0,\phi)|~\mathrm{d}\phi\nonumber\\&\quad\quad\qquad~+\pi\sin^2\beta^{\epsilon}\norm{\psi}_{L^{\infty}(S^2)} \sup_{\theta\in[\theta_0,\beta^{\epsilon}]}\Big|\frac{1}{w^\epsilon(\theta)}-\frac{1}{w^\epsilon(\beta^\epsilon)}\Big|.
\end{align}
Using again $\lim_{\epsilon\to0}|\Omega_2^{\epsilon}|=0,~ \lim_{\epsilon\to0}\beta^{\epsilon}=\theta_0$, it yields that 
\begin{equation}
\lim_{\epsilon\to0} \langle \varrho^\epsilon,\, \psi\rangle=\langle \varrho,\, \psi\rangle,
\end{equation}
namely,  $\varrho^\epsilon$ converges to $\varrho$ vaguely.

5.~As for $\wp^\epsilon$ and $\wp,$ for any $\psi\in C(S^2),$ one has 
\begin{align}
& \langle\wp^\epsilon,~\psi\rangle=\int_{\Omega_1^{\epsilon}}\psi(\theta,~\phi)p_0~\mathrm{d}A+\int_{\Omega_2^\epsilon}\psi(\theta,~\phi)p^\epsilon(\theta)~\mathrm{d}A,~\\
& \langle\wp,~\psi\rangle=0.
\end{align}
Recalling that $p=\frac{\epsilon}{1+\epsilon}\rho(E-\frac{1}{2}|V|^2), E=E_0$ and $w^{\epsilon}(\theta)>w^{\epsilon}(\beta^\epsilon)=\cos\beta^\epsilon$, it holds that 
\begin{equation}
\begin{aligned}
&\quad\quad~\Big|\langle\wp^\epsilon,~\psi\rangle- \langle\wp,~\psi\rangle\Big|\\&=\Big|\int_{\Omega_1^{\epsilon}}\psi(\theta,~\phi)p_0~\mathrm{d}A+\int_{\Omega_2^\epsilon}\psi(\theta,~\phi)p^\epsilon(\theta)~\mathrm{d}A\Big|\\&\leq\norm{\psi}_{L^{\infty}(S^2)}\Big|\Omega\Big|p_0+\frac{\epsilon}{1+\epsilon}\int_{\Omega_2^{\epsilon}}\psi(\theta,\phi)\rho^{\epsilon}\Big|E_0-\frac{1}{2}|V|^2\Big|\,\mathrm{d}A\\&\leq \norm{\psi}_{L^{\infty}(S^2)}\Big|\Omega\Big|p_0+ \frac{\epsilon}{1+\epsilon} \frac{E_0\norm{\psi}_{L^{\infty}(S^2)}}{\cos\beta^\epsilon}\int_{\Omega_2^{\epsilon}}\rho^{\epsilon}w^{\epsilon}\,\mathrm{d}A\\&\leq \norm{\psi}_{L^{\infty}(S^2)}\Big|\Omega\Big|p_0+\frac{\epsilon}{1+\epsilon}\pi E_0\norm{\psi}_{L^{\infty}(S^2)}\tan\beta^\epsilon\sin\beta^\epsilon.
\end{aligned}
\end{equation}
Since $\lim_{\epsilon\to0}p_0=0$ (see \eqref{2.6}) and $\lim_{\epsilon\to0}\beta^\epsilon=\theta_0,$ we infer that 
\begin{equation}
\lim_{\epsilon\to0} \langle \wp^\epsilon,\, \psi\rangle=\langle \wp,\, \psi\rangle,
\end{equation}
thus $\wp^\epsilon$ converges to $\wp$ vaguely.

6.~Moreover, by $\lim_{\epsilon\to0}\beta^\epsilon=\theta_0$ and $\eqref{Shockcondition}$, we have 
	\begin{equation}
	\lim_{\epsilon\to0}p^\epsilon(\theta_0)=\sin^2\theta_0=W_C.
	\end{equation}
The proof is complete. 
\end{proof}

Thus we arrive at the main result of this section.
\begin{theorem}
    For uniform incoming supersonic flow past a straight right-circular cone with half vertex-angle $\theta_0\in(0,\pi/2)$, if the Mach number of the incoming flow $M_{0}$ is suitably large (i.e., $\epsilon$ is suitably small), then there is a  piecewise smooth conical weak entropy solution to problem \eqref{E1}, \eqref{Initialdata} and \eqref{Boundarycondition}. Furthermore, as $M_{0}$ reaches infinity, the piecewise smooth conical weak entropy solution converges in the sense of Lemma \ref{lem41} to the Radon measure solution of the limiting-hypersonic conical flow, which is given by \eqref{49} and \eqref{4.10}.
\end{theorem}

\section{Supersonic  Chaplygin gas passing straight cones}\label{sec5}
In this section, we consider isentropic supersonic Chaplygin gas passing a straight right-circular cone. The governing equations are 
\begin{equation}\label{EE1}
\left\{
\begin{aligned}
& \mathrm{Div}(\rho V)=0, \\
& \mathrm{Div}(\rho V\otimes V)+\mathrm{Grad}\, p=0,\\
& p=A-B\rho^{-1},
\end{aligned}
\right.
\end{equation}
with $A, B$ being given positive constants. The other notations, such as $\mathrm{Div}, \mathrm{Grad}, \rho, V$ represent the same as those in \eqref{E1},\eqref{Initialdata} and \eqref{Boundarycondition}. 
The supersonic uniform incoming flow is
\begin{equation}\label{II1}
(\rho, V)|_{x_1=0}=(\rho_{\infty},(V_{\infty},0,0)),
\end{equation}
and the slip boundary condition reads 
\begin{equation}\label{BB1}
V\cdot \vec{n}|_{\mathcal{C}}=0.
\end{equation}

Also, similar to the polytropic gases, the governing equations for conical Chaplygin flow passing a straight right-circular cone is reduced to \eqref{eq1}-\eqref{eq4}, with boundary conditions \eqref{NewBoundarycondition}.  The definition of weak solutions and Radon measure solutions are the same as  Definitions \ref{Def1} and  \ref{Def2},  just ignoring the equations and terms related to $E$. 

\begin{proposition}\label{Prop2}
	For supersonic Chaplygin gas (i.e. $M_{0}\doteq\frac{\rho_{\infty}V_{\infty}}{\sqrt{B}}\i>1$) passing a straight right-circular cone with half vertex angle $\theta_0$, suppose that there is a piecewise smooth weak solution with a leading conical discontinuous surface (denoted by $\mathcal{D}\mathcal{S}$) attaching at the given cone vertex point, then the half vertex-angle of $\mathcal{D}\mathcal{S}$ is $\beta_0=\arctan\frac{1}{\sqrt{M^2_0-1}}$. It is independent of $\theta_0.$ 
Furthermore, it holds that $$\lim_{M_{0}\to1/\sin\theta_0-}\beta_0=\theta_0.$$ 
\end{proposition}
\begin{proof}
	Since the incoming flow is supersonic, the flow ahead of the discontinuous surface is still the incoming states. Thus, from the Rankine-Hugoniot jump conditions, one gets  
	\begin{equation}\label{EE2}
	\left\{
	\begin{aligned}
	&p_{+}+\rho_{+} u_{+}^2  =p_0+\rho_0 u_0^2\\
	&-1+\rho_{+}^2u^2M_\infty^2  =-\rho_{+}+\rho_{+} M_\infty^2\sin^2\beta_0,\\
	&(\rho_{+}-1) =(\rho_{+}-1)M_\infty^2\sin^2\beta_0,\\
	&w_{+}=w_{0},
	\end{aligned}
	\right.
	\end{equation}
	where $\rho_{+}$ is the limit of $\rho$ by approaching $\mathcal{D}\mathcal{S}$ from the region bounded by $\mathcal{D}\mathcal{S}$ and the cone, and similar definition applied to $p_{+}=A-\frac{B}{\rho_{+}},\, u_{+},\, w_{+}$.  Suppose that $\rho_{+}=\rho_0=1,$ then we have $u_{+}=u_0, \, w_{+}=w_{0},$ which conflicts that $\mathcal{D}\mathcal{S}$ is a discontinuity of the flow field. Thus, $\rho_{+}\neq1.$ Therefore, it follows from  $(\ref{EE2})_3$ that $M_{\infty}\sin\beta_0=1,$ that is $\beta_0=\arctan\frac{1}{\sqrt{M^2_{\infty}-1}}.$ Then, it is obvious that $\lim_{M_0\to1/\sin\theta_0}\beta_0=\theta_0.$ The proof is complete.
\end{proof}

Proposition \ref{Prop2} illustrates that it is not reasonable to define weak solution with a leading conical discontinuous surface for $M_{0}\geq\frac{1}{\sin\theta_0}.$ We thus consider the problem within the framework of Radon measure solution. Indeed, set 
\begin{align}
	& m_a=\rho_0u_0 \mathrm{d}A\mres\Omega+W_a(s)\delta_C,~ ~~~~~~~~~~~ n_a=\rho_0w_0 \mathrm{d}A\mres\Omega+w_a(s)\delta_C,\\
	& m_r=\rho_0u_0w_0 \mathrm{d}A\mres\Omega+W_r(s)\delta_C,~ ~~~~~~~~ n_r=\rho_0w_0^2 \mathrm{d}A\mres\Omega+w_e(s)\delta_C,\\
	& m_t=\rho_0u_0\otimes u_0 \mathrm{d}A\mres\Omega+W_t(s)\delta_C,~ ~~~~~~ n_t=\rho_0|u_0|^2 \mathrm{d}A\mres\Omega+w_t(s)\delta_C,\\
	& \varrho=\rho_0 \mathrm{d}A\mres\Omega+w_\rho\delta_C,~ ~~~~~~~~~~~~~~~~~~~~~~ \wp=-\frac{1}{\rho_0M_\infty^2} \mathrm{d}A\mres\Omega,
\end{align}
and substitute them into Definition \ref{Def2} (ignoring terms related to $E$), by a similar procedure as in \cite[Section 4]{QY} for polytropic gases, we obtain the following results.
\begin{theorem}
	For uniform supersonic Chaplygin gas (i.e. $M_{0}\doteq\frac{\rho_{\infty}V_{\infty}}{\sqrt{B}}\i>1$) passing a straight right-circular cone with half vertex-angle $\theta_0$, suppose that $M_{0}\geq\frac{1}{\sin\theta_0}.$ Then there is a  Radon measure solution, which is given by 
	\begin{align}
	&\varrho=\rho_0 \mathrm{d}A\mres\Omega+\frac{1}{2}\tan\theta_0\delta_C,~\quad \tilde{u}=\tilde{u}_0I_\Omega, \label{Chaplygin1} \\
	& w=w_0I_\Omega+\cos\theta_0I_C, \label{Chaplygin2}
	\end{align}
	with $(\rho_0,\tilde{u}_0, w_0)$ specified by \eqref{upcoming}. Moreover, the pressure on the cone surface is
	\begin{equation}
	W_C=\sin^2\theta_0-\frac{1}{\rho_0M_\infty^2}.\label{Chaplygin3}
	\end{equation}
\end{theorem}

\section*{Acknowledgments}
This work is supported by the National Natural Science Foundation of China under Grants No.11871218, No.12071298, and by Science and Technology Commission of Shanghai Municipality under Grant No.22DZ2229014.

\bibliographystyle{siam}
\bibliography{CKWXHypersonicLimit}



\end{document}